\documentclass{scrarticle}

\usepackage[utf8]{inputenc}
\usepackage[a4paper]{geometry}
\usepackage[colorlinks=true,linkcolor=teal,citecolor=teal,hypertexnames=false]{hyperref}
\usepackage{amsmath,amssymb,amsthm}
\usepackage{enumitem}
\usepackage{microtype}
\usepackage[english]{babel}
\usepackage{mathrsfs}

\usepackage{tikz}
\usetikzlibrary{matrix,arrows}
\usepackage{tikz-cd}
\usetikzlibrary{cd}
\usepackage{comment}
\usepackage[all]{xy}

\usepackage{calrsfs}
\DeclareMathAlphabet{\pazocal}{OMS}{zplm}{m}{n}

\newcommand\bF{\mathbb F}

\newcommand\ZZ{\mathbb Z}
\newcommand\FF{\mathbb F}
\newcommand\Fq{\FF_q}
\newcommand\Spec{\mathrm{Spec}~}
\newcommand\cE{\mathcal{E}}

\newcommand{\ens}{ \ \vert \ }

\newcommand\p{\mathfrak p}
\newcommand\Fp{\FF_{\p}}
\newcommand\Fpm{\FF_{\p_m}}
\newcommand\q{\mathfrak q}

\newcommand\tors{\mathrm{tors}}
\newcommand\End{\operatorname{End}}

\newcommand{\phiT}{\phi^{\sim}}
\newcommand{\expT}{\exp^{\sim}}
\newcommand{\logT}{\log^{\sim}}

\newcommand{\phibar}{\overline{\vphantom{x}\smash{\phi}}}

\newcommand\M{\mathbf M}
\newcommand\Nrd{N_{\mathrm{rd}}}

\newcommand\nil{\textrm{nil}}

\newtheorem{thm}{Theorem}[section]
\newtheorem*{thm*}{Theorem}
\newtheorem{thmA}{Theorem}

\newtheorem{lem}[thm]{Lemma}

\newtheorem{prop}[thm]{Proposition}
\newtheorem{cor}[thm]{Corollary}
\theoremstyle{definition}
\newtheorem{deftn}[thm]{Definition}
\newtheorem*{deftn*}{Definition}
\newtheorem{rem}[thm]{Remark}
\newtheorem*{rem*}{Remark}

\newtheorem{ex}[thm]{Example}

\title{Wieferich primes for Drinfeld modules}

\author{Xavier Caruso, Quentin Gazda, Alexis Lucas}

\begin{document}

\maketitle

\setcounter{tocdepth}{2}

\begin{abstract}
The aim of this paper is to discuss the notion of Wieferich primes in the context of Drinfeld modules.
Our main result is a surprising connection between the proprety of a monic irreducible polynomial $\p$ to be Wieferich and the $\p$-adic valuation of special $L$-values of Drinfeld modules.
This generalizes a theorem of Thakur for the Carlitz module. 

We also study statistical distributions of Wieferich primes, proving in particular
that a place of degree $d$ is Wieferich with the expected probability $q^{-d}$ when we average over large enough sets of Drinfeld modules.
\end{abstract}

\tableofcontents

\section{Introduction}

Let $p$ be an odd prime number. It is said that $p$ is \emph{Wieferich} if the multiplicative order of $2$ is the same modulo $p$ and modulo $p^2$.
Wieferich primes were introduced around $1909$ in \cite{wieferich} by Arthur Wieferich, then a $25$ years old student at M\"unster Universit\"at, who shown that the existence of a non trivial solution to \emph{the first case at $p$ of Fermat's Last Theorem}\footnote{A triple $(x,y,z)\in \mathbb{Z}^3$ is said to satisfy \emph{the first case at $p$ of Fermat's Last Theorem} if the product $xyz$ is prime to $p$ and $x^p+y^p=z^p$.}, $p\geq 5$, would imply that $p$ is Wieferich.
The search for Wieferich primes went up to primes $<2^{64}$ yet only two were discovered: the very first one, $1093$, was found by Meissner \cite{meissner} and the second one, $3511$, by Beeger \cite{beeger}.

Replacing the number $2$ by an arbitrary number $a$, we obtain the notion of \emph{Wieferich primes in base $a$}. Despite the surprising lack of datas on these prime numbers, it is expected that the number of Wieferich primes below $N$ (with respect to a fixed base) is of order $\log(\log N)$; in particular, there should be infinitely many Wieferich primes. But it does not even seem known whether there exists at least one Wieferich prime number in each base.
We refer the reader to \cite{katz} for the state of the art on Wieferich primes.

Let now $\bF$ be a finite field with $q$ elements and let $C$ denote the \emph{Carlitz module};
it is the counterpart of the multiplicative group scheme in function field arithmetic (\emph{e.g.} see~\cite{gazda-junger} for details on this statement). 
Taking profit of the latter analogy, Thakur introduced in~\cite{thakur} the notion of \emph{Wieferich primes} in function field arithmetic.

\begin{deftn*}[Thakur]
Let $\p$ be an irreducible monic polynomial (also called \emph{prime} below). We say that $\mathfrak{p}$ is \emph{$C$-Wieferich} if the kernel of $(a\mapsto C_a(1) \bmod{\p})$ in $\bF[t]$ agrees with that of $(a\mapsto C_a(1) \bmod{\p^2})$. 
\end{deftn*}

As for the classical notion, very little is known on $C$-Wieferich primes. The problem of the infiniteness of $C$-Wieferich primes is probably as hard as the classical problem in number theory. Nonetheless, Thakur proved the following surprising connection with $p$-adic zêta functions in \emph{loc.\,cit.} 
\begin{thm*}[Thakur]\label{thm:thakur}
We assume $q > 2$. The prime $\p$ is $C$-Wieferich if and only if $\p$ divides the $\p$-adic Carlitz zêta value $\zeta_{\p}(1)$.
\end{thm*}

\begin{rem*}
It is worth to mention that, inspired by the function field situation, Thakur proved a similar result in classical arithmetic. Let $p\geq 5$ be a prime number. Then Thakur shows in the same paper that $p$ is \emph{Wilson}\footnote{A prime number $p$ is called \emph{Wilson} if $(p-1)!\equiv -1\pmod{p^2}$.} if and only if $p$ divides in \emph{the $p$-adic Euler-Mascheroni constant}\footnote{Recall that \emph{the $p$-adic Euler-Mascheroni constant} $\gamma_p$ is defined as the $p$-adic limit
\[
\gamma_p:=\lim_{\varepsilon\to 0} \left(\zeta_p(1+\varepsilon)-\frac{1-1/p}{\varepsilon}\right)
\]
where $\zeta_p$ denotes the $p$-adic zêta function.} $\gamma_p\in \mathbb{Z}_p$.
The relation with Theorem \ref{thm:thakur} is the following: his theory of function fields $\Gamma$-values enables Thakur to determine the counterpart of Wilson primes in function fields arithmetic--called $C$-Wilson primes--and then prove that a prime $\p$ is $C$-Wilson if, and only if, it is $C$-Wieferich. 
\end{rem*}

\paragraph{Objectives and results.}

The Carlitz module is the simplest example of a Drinfeld module. If the Carlitz module is the analogue to the multiplicative group scheme in classical arithmetic, Drinfeld modules are the counterparts of elliptic curves (again, we refer the reader to \cite{gazda-junger} for the relevant statements). The goal of this paper is to generalize Thakur's theorem to some more general Drinfeld modules. 

Let $\phi:\FF[t]\to \FF[t]\{\tau\}$, $a \mapsto \phi_a$ be a (model of a) Drinfeld module over $\FF[t]$ (we refer to Subsection \ref{sec:background} and Appendix~\ref{sec:models} for the definitions).
We begin by generalizing Thakur's definition.

\begin{deftn*}[see Definition~\ref{def:wieferich}]
For an ideal $I$ of $\FF[t]$, we let $\pi(\phi;I)$ be the annihilator of~$1$ modulo $I$ for the $\FF[t]$-action of $\phi$.
We say that a prime $\p\in \FF[t]$ is \emph{$\phi$-Wieferich} if $\pi(\phi;\p)=\pi(\phi;\p^2)$.

More generally, we denote by $c_{\p}(\phi)$ the maximal integer $c$ for which $\pi(\phi;\p)=\pi(\phi;\p^{c+1})$. 
\end{deftn*}

If $\phi=C$ is chosen to be the Carlitz module over $\FF[t]$, we recover Thakur's definition of $C$-Wieferich primes.
Clearly, $c_{\p}(\phi)$ is positive if and only if $\p$ is $\phi$-Wieferich. However, the datum of $c_{\p}(\phi)$ is more precise; in some sense, it measures the ``order at which $\p$ is Wieferich''. 

On the other hand, to any Drinfeld module $\phi$ as above, one may also associate a special $L$-value $L(\phi;1)$ converging in $\FF[\![t^{-1}]\!]$. When $\phi$ is the Carlitz module, $L(C,1)$ corresponds to the Carlitz zeta value at $1$. 
A $\p$-adic variant also exists: it leads to the definition of the $\p$-adic special $L$-value, denoted by $L^*_{\p}(\phi;1)$, converging in the $\p$-completion of $K$; see Subsection \ref{sec:background} below for details.

In \cite{caruso-gazda}, the first and the second author established an algorithm based on Anderson's trace formula to compute $L^*_{\p}(\phi;1)$ efficiently. Computations in mass using this algorithm suggested an equality between the $\p$-adic valuation of $L^*_{\p}(\phi;1)$ and $c_\p(\phi)$ in some cases, but not all. The results of the numerous calculations helped to pin down the precise condition on $\phi$ for when this shall happen.

\begin{deftn*}[see Defintion~\ref{def:small}]
Write $\phi_t=t+g_1\tau+\ldots+g_r\tau^r$.
We say that $\phi$ is \emph{small} (resp. \emph{very small}) if $\deg_{t} g_i\leq q^i$ (resp. $\deg_{t} g_i < q^i$) for all $i\in \{1,\ldots,r\}$.
\end{deftn*}

In the present article, we confirm the hypothesis that emerged after our campaign of computations, proving the following theorem.

\begin{thmA}
\label{thm:smallWieferich}
Let $\p$ be a prime of $\FF[t]$. If $q = 2$, we further assume that $\deg \p > 1$. 
\begin{enumerate}[label=(\theenumi)]
\item If $\phi$ is very small, then $v_\p\big(L_\p(\phi; 1)\big) = c_\p(\phi)$.\\
In particular, in this case, $\p$ is $\phi$-Wieferich if and only if $\p$ divides $L_\p(\phi; 1)$.
\item If $\phi$ is small and $1$ is not a torsion point, then
$v_\p\big(L_\p^*(\phi; 1)\big) = c_\p(\phi)$.\\
In particular, in this case, $\p$ is $\phi$-Wieferich if and only if $\p$ divides $L_\p^*(\phi; 1)$.
\end{enumerate}
\end{thmA}

\begin{rem*}
We will actually prove a more general version of Theorem~\ref{thm:smallWieferich} without the smallness assumption (see Theorems~\ref{thm:main} and~\ref{thm:vpLp});
in this case, however, the base-point~$1$ considered when defining $c_\p(\phi)$ has to be replaced by another value (a certain Taelman unit).
\end{rem*}

\begin{rem*}
We--the authors--are not aware of a similar statement in classical arithmetic involving, \emph{e.g.} models of elliptic curves. 
\end{rem*}

The proof of Theorem \ref{thm:smallWieferich} relies on a twisted log-algebraicity result of Anglès--Ngo Dac--Tavares Ribeiro \cite{ANT}.
As a byproduct of our work, we also prove Theorem~\ref{thm:order} below which provides a positive answer to an analogue of the main conjecture of~\cite{caruso-gazda} in the framework of Drinfeld modules
(we refer to Appendix~\ref{app:anderson} for a comparison between the language of~\cite{caruso-gazda} and that of the present paper).

\begin{thmA}
\label{thm:order}
Let $\phi : \mathbb F[t] \to \mathbb F[t]\{\tau\}$ be a Drinfeld module.
The order of vanishing at $T=1$ of the $L$-series $L_\p(\phi; T)$ is independent of $\p$.
\end{thmA}

We underline that Theorem~\ref{thm:order} holds without any smallness assumption.

Finally, in Section \ref{sec:statistics}, with the perspective of investigating the infiniteness of Wieferich primes, we study their repartition from a probabilistic point of view.
Indeed, the classical heuristic says that a prime $\p$ should be $\phi$-Wieferich with probability $q^{-\deg \p}$, supporting eventually the hypothesis that $\phi$ admits infinitely many Wieferich primes.
Our main result is a full justification of the above heuristic when we average over all small Drinfeld modules of sufficiently large rank.

\begin{thmA}
\label{thm:proba}
Let $r$ and $d$ be two positive integers.
\begin{enumerate}[label=(\theenumi)]
\item 
We assume $r \geq d + \log_q(2d)$.
Then, for any fixed place $\p$ of degree $d$, the probability that $\p$ is $\phi$-Wieferich is $q^{-d}$ when $\phi$ varies among small Drinfeld modules of rank at most $r$.
\item
We assume $r \geq 2d$.
Then, the above events are mutually independent when $\p$ varies among all places of degree $d$.
\end{enumerate}
\end{thmA}

We mention nonetheless that Theorem~\ref{thm:proba} does not have any concrete implication on the infiniteness of Wieferich primes for a \emph{fixed} Drinfeld module;
it actually even cannot ensure the existence of a single Drinfeld module admitting an infinite number of Wieferich places.
We conclude our investigations by numerical experiments from which, unfortunately, no clear conclusion emerges.

\medskip

All along this text, we keep a limited setting: our coefficient ring is $\FF[t]$, our base field is $\FF(t)$ and we constrain ourselves to Drinfeld modules, \emph{i.e.} dimension one. We plan to pursue our investigations beyond this restricted setting in a future work. 

\paragraph{Notations.}

We fix a finite field $\mathbb F$ and denote by $p$ and $q$ its characteristic and cardinality respectively.
We set $A := \mathbb F[t]$ and $K := \mathbb F(t)$.

Throughout the article, we will often slightly abuse notation and use the same letter to denote an ideal of $A$ and its monic generator.
We hope that this will not cause confusion.

If $\p$ is a place of $A$, that is a monic irreducible polynomial, we let $v_\p$ denote the associated $\p$-adic valuation.
We write $A_\p$ (resp. $K_\p$) for the completion of $A$ (resp. $K$) for the $\p$-adic topology.
We recall that $A_\p$ is a local ring and we denote by $\mathfrak m_\p$ its maximal ideal.

For a (commutative, unitary) $\FF$-algebra $R$, we denote by $R\{\tau\}$ the noncommutative $R$-algebra of \emph{twisted polynomials}, \emph{i.e.} polynomials in $\tau$ with coefficients in $R$ subject to the twisted multiplication rule $\tau c = c^q \tau$. Note that $R\{\tau\}$ acts on $R$ via $\tau\cdot c:=c^q$.

We also let $R\{\!\{\tau\}\!\}$ be the ring of twisted power series, where elements are given by infinite power series in $\tau$, subject to the same commutation relation.

\paragraph{Acknowledgments.} We wish to express our gratitude to Bruno Anglès and Floric Tavares Ribeiro for enlightening discussions and several invitations to the LMNO. We are also grateful to the ANR PadLefan for traveling support at several stages of the project. The second and third authors are also indept to the Journal of Number Theory for funding support to the AMS-UMI Meeting in Palermo where some of the results of the paper were presented.

\section{Wieferich primes and $L$-series}

This section is devoted to introduce Wieferich type properties in the framework of Drinfeld modules and study their relationships with the theory of $L$-series.

\subsection{Review on Drinfeld modules and their $L$-series}\label{sec:background}

In this preliminary subsection, we introduce the main objects we shall work with throughout this article.
In particular, we recall Anglès, Ngo Dac and Tavares--Ribeiro's construction of $T$-twisted Drinfeld modules and their applications to the calculations of $L$-series,
as this theory will appear as a crucial ingredient in our proofs.

\subsubsection{Drinfeld models}

We start by giving the definition of Drinfeld modules we will use.
Since we are working over $A$, which is not a field, we prefer using the word ``model'' instead of ``module'';
we refer to Appendix~\ref{sec:models} for a discussion about this choice and a comparison with other more geometric definitions that can be found in the literature.

\begin{deftn}
Let $r$ be a positive integer. 
A \emph{model of a Drinfeld $A$-module} of rank $r$ over $A$, or simply \emph{Drinfeld model} for short, is an $\FF$-algebra homomorphism $\phi:A\to A\{\tau\}$, $a \mapsto \phi_a$ such that $\phi_t$, as a polynomial in $\tau$, has degree $r$ and constant term $t$. 
\end{deftn}

If $R$ is an $A$-algebra, we denote by $\phi(R)$ the $A$-module which, as an $\FF$-vector space is just $R$, and where $A$ acts through $\phi$. When $R$ is finite, $\phi(R)$ is a finite $A$-module and we let $|\phi(R)|$ denotes (the monic generator of) its Fitting ideal. 

Given a Drinfeld model $\phi$, there is a unique noncommutative formal power series
\begin{equation}\label{eq:exponential}
\exp_{\phi}:=e_0+e_1\tau+e_2\tau^2+\cdots \in \FF(t)\{\!\{\tau\}\!\}
\end{equation}
satisfying $e_0=1$ and $\phi_t\cdot \exp_\phi=\exp_\phi \cdot\,t$ (\emph{cf}. \cite[\S 4.2]{goss}). We call $\exp_\phi$ the \emph{exponential} of~$\phi$. Dually, there exists a unique noncommutative formal power series
\[
\log_\phi:=\ell_0+\ell_1\tau+\ell_2\tau^2+\cdots \in \FF(t)\{\!\{\tau\}\!\}
\]
satisfying $\ell_0=1$ and $\log_\phi\cdot\,\phi_t=t\cdot \log_\phi$. We have $\exp_\phi \cdot \log_\phi=\log_\phi\cdot \exp_\phi=1$, which justifies to call $\log_\phi$ the \emph{logarithm} of $\phi$.

The logarithm has also a $\p$-adic incarnation.
Before explaining it, we recall that $A_\p$ and $K_\p$ denote the $\p$-adic completions of $A$ and $K$ respectively,
and that $\mathfrak m_p$ is the maximal ideal of $A_\p$.
In addition, we make the following hypothesis:
\begin{enumerate}[leftmargin=10ex,label=$(H)_{\p}$]
\item\label{Hyp}: \, if $q=2$, then $\deg \p>1$.
\end{enumerate}
The following proposition is proved in \cite[Subsection~3.4]{AT} (see in particular Lemma~3.20).

\begin{prop}\label{prop:isom-p-adic-log}
We assume \ref{Hyp}.
Then, the formal series $\log_\phi$ converges over $\mathfrak{m}_\p$ and defines a bijective isometry of $\mathfrak{m}_\p$.
\end{prop}

\begin{rem}
Proposition \ref{prop:isom-p-adic-log} remains valid without the assumption \ref{Hyp} at the cost of replacing $\mathfrak{m}_\p$ by $\mathfrak{m}_\p^2$.
\end{rem}

To avoid confusion, we will write $\log_{\phi,\p}$ for the $\p$-adic analytic function defined by $\log_\phi$.
We call it the \emph{$\p$-adic $\phi$-logarithm} and we extend it to a function $\log_{\phi,\p} : A_\p \to K_\p$ by setting:
\begin{equation}\label{eq:p-adic-log}
\log_{\phi,\p}(x):=\frac{1}{|\phi(\FF_\p)|}\cdot \log_{\phi,\p}\big(\phi_{|\phi(\FF_\p)|}(x)\big)
\end{equation}
which makes sense because $\phi_{|\phi(\FF_\p)|}(x)\in \mathfrak{m}_\p$.
It is actually sometimes convenient to consider $\log_{\phi,\p}$ with domain $\phi(A_\p)$ (which, we recall is just $A_\p$ endowed with the structure of $A$-module given by $\phi$) because it then becomes $A$-linear. 

\begin{lem}
\label{lem:torsion}
We assume \ref{Hyp}. Then $\ker\left(\log_{\phi,\p}:\phi(A)\to K_{\p}\right)= \phi(A)_{\mathrm{tors}}$.
\end{lem}

\begin{proof}
We consider $x\in \phi(A)_{\operatorname{tors}}$ and $a\in A\backslash\{0\}$ such that $\phi_a(x)=0$.
Then 
\[a \log_{\phi,\p}(x)= \log_{\phi,\p}\phi_a(x)=0\]
from what it follows that $\log_{\phi,\p}(x)$ vanishes.
Conversely, let $x\in A$ such that $\log_{\phi,\p}(x)=0$. Then, taking $a = |\phi(\FF_\p)|$, we have
$\log_{\phi,\p}(\phi_a(x)) = a \log_{\phi,\p}(x) = 0$.
Since $\phi_a(x) \in \mathfrak{m}_\p$, we deduce from Proposition~\ref{prop:isom-p-adic-log} that $\phi_a(x)=0$. So $x\in \phi(A)_\tors$.
\end{proof}

\subsubsection{$T$-twisted versions}
\label{sssec:Ttwisted}

Let $T$ be a new formal variable. 
Following Anglès--Ngo Dac--Tavares Ribeiro in \cite{ANT}, we write $\phi_t=t+g_1 \tau+\ldots +g_r \tau^r$ and set 
\[
\phiT_t:=t +g_1 T\tau+\ldots + g_r T^{r} \tau^r \in A[T]\{\tau\}:=A\{\tau\}\otimes_{\FF}\FF[T]
\]
(so that the variables $\tau$ and $T$ commute). This allows to define another object, playing the role of a $T$-deformation of $\phi$, which is tightly related to the $L$-series of $\phi$.

\begin{deftn}\label{def:Ttwisted}
The \emph{$T$-twisted form of $\phi$} is the $\FF$-algebra homomorphism $\phiT:A\to A[T]\{\tau\}$ induced by $t \mapsto \phiT_t$.
\end{deftn}

\begin{rem}
Note that we recover $\phi$ via the composition $\phiT:A\to A[T]\{\tau\}\to A\{\tau\}$ mapping $T$ to $1$.
\end{rem}

Many gadgets attached to Drinfeld models admit a canonical $T$-twisted version which recovers the classical version as $T=1$. In particular, given an $A$-algebra $R$, we let $\phiT(R)$ denote $R[T]$ endowed with its $A[T]$-module structure induced by $\phiT$. 
Let also $|\phiT(R)|\subset R[T]$ be its Fitting ideal.
If $R = F$ is a finite extension of $\FF$, we observe that $\phiT(F)$ admits the free $A[T]$-linear resolution
\begin{equation}
\label{eq:resolutionphiT}
0\longrightarrow A \otimes_{\FF} F[T] \xrightarrow{t-\phiT_t} A \otimes_{\FF} F[T] \xrightarrow{a \otimes x\:\mapsto\:\phiT_a(x)} \phiT(F) \longrightarrow 0.
\end{equation}
Hence $|\phiT(F)|$ is the principal ideal generated by
\[
\mathrm{det}_{A[T]} \big(t-\phiT_t~|~A \otimes_{\FF} F[T]\big)\in A[T].
\]
In what follows, we will often abuse notation and continue to write $|\phiT(F)|$ for the previous generator.

We define as well the \emph{$T$-twisted exponential} of $\phi$ as the (noncommutative) power series 
\begin{equation}\label{eq:TtwistedExpE}
\expT_\phi:=e_0+e_1 T\tau+e_2 T^2\tau^2+\ldots \in K[T]\{\!\{\tau\}\!\}
\end{equation}
where $K[T]\{\!\{\tau\}\!\}$ is understood as the noncommutative ring where $\tau c=c^q\tau$ ($c\in K$) and $\tau T=T\tau$, and where the coefficients $e_i\in K$ are the same than in Equation~\eqref{eq:exponential}. Note that we have 
$\phiT_{a}\cdot \expT_\phi=\expT_\phi\cdot\,a$
for all $a\in A$ (see \cite[\S 2.2]{ANT} for details).

Similarly, the \emph{$T$-twisted logarithm} of $\phi$ is the series:
\[
\logT_\phi:=\ell_0+\ell_1 T\tau+\ell_2 T^2\tau^2+\ldots \in K[T]\{\!\{\tau\}\!\}.
\]
Again, it satisfies $\logT_{\phi}\cdot\,\phiT_{a}=a\cdot \logT_{\phi}$ for all $a\in A$ and we have
\[
\expT_{\phi}\cdot \logT_{\phi}=\logT_{\phi}\cdot \expT_{\phi}=1.
\] 
A $\p$-adic incarnation of the $T$-twisted logarithm also exists. To define it, we consider the Tate algebra over $K_\p$ in the variable $T$:
\[
K_\p\langle T\rangle=\left\{\, \sum_{n = 0}^\infty a_n T^n\ens a_n\in K_\p \text{ and } \lim\limits_{n\rightarrow \infty}v_\p(a_n)=0\, \right\}
\]
Similarly, we denote by $A_\p\langle T\rangle$ (resp. $\mathfrak m_\p\langle T\rangle$) the subspace of $K_\p\langle T\rangle$ consisting of power series with coefficients in $A_\p$ (resp. in $\mathfrak m_\p$).
The formal series $\logT_{\phi}$ converges over $\mathfrak{m}_{\mathfrak{p}}\langle T \rangle$ and we extend it to a function
\[
\logT_{\phi,\p}:A_{\mathfrak{p}}\langle T \rangle \longrightarrow |\phiT(\FF_\p)|^{-1}\cdot K_{\mathfrak{p}}\langle T\rangle
\]
by setting
\[
\logT_{\phi,\p}(f):=\frac{1}{|\phiT(\FF_\p)|} \cdot \logT_{\phi}(\phiT_{|\phiT(\FF_\p)|}(f)).
\]
(compare with Equation~\eqref{eq:p-adic-log}).
Be careful that $|\phiT(\FF_\p)|$ is not invertible in $K_{\mathfrak{p}}\langle T\rangle$ because it is a polynomial in $T$ whose constant coefficient is $\mathfrak{p}$ and whose leading coefficient is a unit (see below).

\subsubsection{$L$-series and class formula}

Let $\phi$ be a Drinfeld model and $\phiT$ be the $T$-twisted form of $\phi$ as before.
In what follows, we attach an $L$-series to $\phi$ out of $\phiT$, following \cite{ANT}.
For this, the first step is to define local factors.
Given a place $\p$ of $A$, the \emph{local $L$-factor} of $\phi$, denoted by $P_{\p}(\phi;T)$, is the following polynomial in $T$:
\begin{equation}\label{eq:localLfactor}
P_{\p}(\phi;T) := \p^{-1} \cdot |\phiT(\FF_{\p})| \in \FF(t)[T].
\end{equation}
Evaluating at $T=0$ and letting $\theta$ denote the image of $t$ in $\FF_\p = \FF[t]/\p$, we get
\[
|\phiT(\FF_\p)|_{|T=0} = \mathrm{det}_{A}\big(t-\theta\,|\,\FF_{\p}[t]\big) = \mathrm{Norm}_{\FF_{\p}[t]/\FF[t]}(t-\theta) = \p.
\]
Hence $P_{\p}(\phi;0) = 1$.
It can be shown more generally that $P_{\p}(\phi;T)$ belongs to $1+T^{\deg \p}K[T]$. In particular, the following product
\begin{equation}\label{eq:formal-L-series}
L(\phi;T):=\prod_{\p} \frac 1 {P_{\p}(\phi;T)},
\end{equation}
taken over all monic irreducible polynomials in $A$, makes sense in $K[\![T]\!]$. 

\begin{deftn}
The series $L(\phi;T)$ is called the \emph{formal $L$-series} of $\phi$.\\
The \emph{formal $\p$-adic $L$-series} of $\phi$ is defined as $L_{\p}(\phi;T):=P_{\p}(\phi;T)L(\phi;T)$.
\end{deftn}

It follows from \cite[Lemma~2.2]{AT} that $L(\phi;T)$ converges on the closed $\infty$-adic unit disk. Therefore, we can evaluate the $L$-series at $T=1$. 
By~\cite[Theorem~3.23]{AT}, we know similarly that $L_{\p}(\phi;T)$ defines an analytic function on $A_\p$. However, in the $\p$-adic case, $L_\p(\phi;T)$ may vanish at $T=1$, so we further define a \emph{special $L$-value} at $T=1$ by factoring out the highest power at $T=1$: we write
\[
L_\p(\phi;T)=(T-1)^k \cdot L^*_\p(\phi;T)
\]
with $L^*_\p(\phi;1)\neq 0$ for a unique integer $k\geq 0$. The special value mentioned above is $L^*_\p(\phi;1)$.

\begin{rem}
Both results of convergence can be recovered by combining Theorem~\ref{thm:compLseries} of Appendix~\ref{app:anderson} together with a straightforward extension of \cite[Theorem~2.2.6]{caruso-gazda}, proving thusly that $L(\phi;T)$ and $L_{\p}(\phi;T)$ have actually infinite radius of convergence.
\end{rem}

There is a close connection between $L$-series and $T$-twisted exponentials and logarithms as defined previously.
Indeed, Anglès, Ngo Dac and Tavares Ribeiro proved in \cite{ANT} that the element
\begin{equation}\label{eq:Ttwisted-class-formula}
u_\phi(T):=\expT_{\phi}\left(L(\phi;T)\right)
\end{equation}
lies in $A[T]$. This is the so-called \emph{$T$-twisted class formula}.
Indeed, taking $T = 1$, we observe that $u_\phi(1)$ is the image of $L(\phi; 1)$ by the classical exponential attached to $\phi$;
Anglès, Ngo Dac and Tavares Ribeiro's result then appears as a generalization of the classical Taelman's class formula~\cite{taelman}.

Applying $\logT_{\phi}$ to Equation~\eqref{eq:Ttwisted-class-formula}, we obtain
\begin{equation}\label{classformulaT}
L_{\p}(\phi;T) = \p^{-1} \cdot \logT_{\phi,\p} \phi_{|\phiT(\FF_\p)|}(u_\phi(T))\in K_{\p}\langle T\rangle
\end{equation}
which, after taking the values at $T=1$ yields the so-called \emph{$\p$-adic class formula}:
\begin{equation}\label{classformula}
L_{\p}(\phi;1) = \p^{-1}\cdot \log_{\phi,\p} \phi_{|{\phi}(\FF_\p)|}(u_\phi(1))\in K_{\p}.
\end{equation}

\subsection{Ordic valuations}

In this subsection, we fix a Drinfeld model $\phi$. We also fix a base-point $x\in A$. The next definition generalizes Thakur's notion of $C$-Wieferich primes to $\phi$.
\begin{deftn}\label{def:pi}\label{def:wieferich}
Given an ideal $I$ of $A$, we define
\[
\pi_x(\phi;I) = \ker\big( A \to A/I, \, a \mapsto \phi_a(x) \big).
\]
A place $\p$ of $A$ is said \emph{$\phi$-Wieferich in base $x$} if
$\pi_x(\phi;\p) = \pi_x(\phi;\p^2)$.
\end{deftn}

\begin{rem}
\label{rem:ind-Fx}
The function $\pi_{x}$ only depends on the $\FF^{\times}$-orbit of $x$. 
\end{rem}

We record some properties of $\pi_x$ as a function of ideals.
\begin{prop}\label{prop:properties-of-pi}
Let $I$ and $J$ be two nonzero ideals in $A$.
\begin{enumerate}[label=$(\arabic*)$]
\item\label{item:inclusion} If $I\subset J$, then $\pi_x(\phi;I)\subset \pi_x(\phi;J)$.
\item\label{item:coprime} If $I$ and $J$ are coprime, $\pi_x(\phi;IJ)=\pi_x(\phi;I)\cap \pi_x(\phi;J)$.
\item\label{item:tour} If $\p$ is a maximal ideal satisfying \ref{Hyp} and $k\geq 2$, then 
  $\pi_x(\phi;\p^k)$ either equals $\pi_x(\phi;\p^{k-1})$ or $\p \pi_x(\phi;\p^{k-1})$;
  besides, if $\pi_x(\phi;\p^k)=\p \pi_x(\phi;\p^{k-1})$ then $\pi_x(\phi;\p^{k+1})=\p \pi_x(\phi;\p^{k})$.
\item\label{item:divise} $\pi_x(\phi;I)$ divides the ideal $|\phi(A/I)|$. 
\end{enumerate}
\end{prop}
\begin{proof}
Point \ref{item:inclusion} is deduced from the $A$-module map $\phi(A/I)\to \phi(A/J)$.
By the Chinese remainder theorem, we have $\phi(A/IJ)\cong \phi(A/I)\times \phi(A/J)$ and point \ref{item:coprime} follows.

We turn to \ref{item:tour}. By \ref{item:inclusion}, we have $\pi_x(\phi;\p^{k})\subset \pi_x(\phi;\p^{k-1})$. We claim that $\p\pi_x(\phi;\p^{k-1})\subset \pi_x(\phi;\p^k)$. Indeed, for $a\in \pi_x(\phi;\p^{k-1})$, we have $\phi_a(x)\in \p^{k-1}$ and hence $\phi_{\p a}(x)=\phi_{\p}(\phi_a(x))\in \p^{k}+(\phi_a(x))^q$; assumptions on $k$ are such that $\p^{q(k-1)}\subset \p^k$, thus $\phi_{\p a}(x)\in \p^k$.

This shows that $\pi_x(\phi;\p^k)$ either equals $\pi_x(\phi;\p^{k-1})$ or $\p\pi_x(\phi;\p^{k-1})$. It remains to prove that if the latter case holds, then $\pi_x(\phi;\p^{k+1})=\p\pi_x(\phi;\p^k)$; for that it is enough to show that the canonical map
\[
\pi_x(\phi;\p^{k+1})/\p \pi_x(\phi;\p^k)\longrightarrow \pi_x(\phi;\p^{k})/\p \pi_x(\phi;\p^{k-1})
\]
is injective. By the Snake lemma, this amounts to showing that the multiplication by $\p$
\[
\pi_x(\phi;\p^{k-1})/\pi_x(\phi;\p^k)\xrightarrow{\times \p} \pi_x(\phi;\p^{k})/\pi_x(\phi;\p^{k+1})
\]
is injective. Given $a\in \pi_x(\phi;\p^{k-1})$ such that $\p a$ belongs to $\pi_x(\phi;\p^{k+1})$, we have $\p \phi_a(x)\in (\phi_{\p a}(x))+(\phi_{\p}(x))^q\subset \p^{k+1}$. If $q>2$, we deduce that $a\in \pi_x(\phi;\p^{k})$ which shows injectivity. If $q=2$, the assumption \ref{Hyp} ensures that the degree of $\p$ is at least $2$; hence by \cite[1.4.(ii)]{gekeler} the $\tau$-valuation of $\phi_{\p} \bmod{\p}$ is at least $2$ as well. We deduce that
\[
\phi_{\p a}(x)=\phi_\p(\phi_a(x))\in \p(\phi_a(x))+\p(\phi_a(x)^2)+(\phi_{a}(x))^{4}\subset \p^{k+1}
\]
from what we get $a\in \pi_x(\phi;\p^{k})$ as before.

Point \ref{item:divise} is deduced from the fact that the Fitting ideal $|\phi(A/I)|$ is contained in the annihilator of $\phi(A/I)$.
\end{proof}

Since the ring $A$ is principal, those properties reduces the determination of the function $\pi_x(\phi;I)$ to that of its value on maximal ideals and the integers $c_{\p}(\phi;x)$ which appear in the next definition.
\begin{deftn}
Let $\p$ be a prime in $A$. 
The \emph{$\p$-ordic valuation} of $x$ relatively to $\phi$, denoted by $c_{\p}(\phi;x)$, is the largest $c \in \{0,1,\ldots,\infty\}$ such that $\pi_x(\phi;\p^{c+1})=\pi_x(\phi;\p)$.
\end{deftn}

That is, a prime $\p$ is $\phi$-Wieferich if and only if $c_{\p}(\phi;x)>0$. Besides, $c_{\p}(\phi;x)=\infty$ if and only if $x$ is a $\phi$-torsion point,
\emph{i.e.} there exists $a \in A$ such that $\phi_a(x) = 0$.

\begin{rem}
We deduce from Proposition~\ref{prop:properties-of-pi} that 
\[
\pi_x(\phi;I)=\bigcap_{i=1}^s\p_i^{\max(0,\,k_i-c_{\p_i}(\phi;x)-1)}\:\pi_x(\phi;\p_i)
\]
if $I=\p_1^{k_1}\cdots \p_s^{k_s}$ is the prime ideal decomposition of $I\neq (0)$. 
\end{rem}

\subsection{Connection with special values of $L$-series}

In this subsection, we present a set of theorems relating ordic valuations to the $\p$-adic valuation of several values of interest attached to the $\p$-adic $L$-series.

\subsubsection{The value at $1$}

To start with, we focus on $L_\p(\phi; 1)$.
We will show that its $\p$-adic valuation is related to $c_{\p}(\phi; x)$ when $x$ is the Taelman unit $u_\phi(1)$ introduced at the end of Subsection~\ref{sec:background}.

\begin{prop}\label{prop:clé} Let $\p$ be a place satisfying \ref{Hyp}, let $x\in A$. Then
\begin{equation}\label{eq:compute-cp}
c_{\p}(\phi;x)=v_{\p}\big(\phi_{|\phi(\FF_\p)|}(x)\big)-1.
\end{equation}
\end{prop}

\begin{proof}
To simplify notations, we set $a := |\phi(\FF_\p)| \in A$ throughout the proof.

We claim that $x$ is a $\phi$-torsion point if and only if $\phi_a(x)=0$.
The direct implication is clear. To prove the converse, we assume that $\phi_a(x) \neq 0$. Since $\phi_a(x) \in \p A$, in virtue of Proposition \ref{prop:isom-p-adic-log}, we may apply the $\p$-adic logarithm and get 
$\log_{\phi,\p}(\phi_a(x))=a \log_{\phi,\p}(x)\neq 0$. In particular, $\log_{\phi,\p}(x)\neq 0$ and according to Lemma~\ref{lem:torsion}, $x$ is not a torsion point.

We now observe that, if $x$ is $\phi$-torsion, then both quantities in~\eqref{eq:compute-cp} are infinite by the above claim and the result is clear.
Thus, we may assume that $x$ is not $\phi$-torsion, so that both quantities are finite. As $\FF$-vector spaces, we have $\phi(A/I)= A/I$ where $I\subset A$ is an ideal.
In particular, $\deg |\phi(A/I)|=\dim_{\FF} \phi(A/I)=\dim_{\FF} A/I=\deg I$.  We deduce that the quotient $a/\pi_x(\phi;\p)$ is of degree smaller than $\deg \p$, hence prime to $\p$.

We also note that, for a given nonnegative integer $k$, the condition 
``$\p^k$ divides $\phi_a(x)$'' is equivalent to ``$\pi_{x}(\phi;\p^k)$ divides $a$''.
Therefore, $v_{\p}(\phi_a(x))$ is the largest integer $k$ such that $\pi_x(\phi;\p^k)$ divides $a$.
Since the quotient $a/\pi_x(\phi;\p)$ is prime to $\p$, Property~\ref{item:tour} of Proposition~\ref{prop:properties-of-pi} gives the equality~\eqref{eq:compute-cp}.
\end{proof}

\begin{thm}
Let $\p$ be a place satisfying \ref{Hyp}. Then
\[
v_{\p} (L_{\p}(\phi;1))=c_{\p}(\phi;u_\phi(1)).
\]
In particular, $\p$ is $\phi$-Wieferich in base $u_{\phi}(1)$ if and only if $\p$ divides $L_{\p}(\phi;1)$.
\end{thm}
\begin{proof}
We apply Proposition \ref{prop:clé} to $u_\phi(1)$ and the class formula \eqref{classformula} to obtain the desired equality. The second assertion follows directly.
\end{proof}

\subsubsection{Special value: the non-torsion case}

We recall that the special value $L_\p^*(\phi; 1)$ is defined by dividing the $L$-series $L_\p(\phi; T)$ by the highest possible power of $T{-}1$ and then taking the value at $1$.
Similarly, we write
\[
u_\phi(T) = (T-1)^k \cdot u_{\phi}^*(T)
\]
with $u^*_{\phi}(1)\neq 0$. The exponent $k$ is then the order of vanishing of $u_\phi(T)$ at $T = 1$.

\begin{thm}\label{thm:main}
We assume that $A$ has no $\phi$-torsion.
\begin{enumerate}[label=(\theenumi)]
\item For any place $\p$, the vanishing orders at $T = 1$ of $L_{\p}(\phi;T)$ and $u_\phi(T)$ are the same.\\
  In particular the vanishing order of $L_{\p}(\phi;T)$ does not depend on~$\p$.
\item For any place $\p$ satisfying \ref{Hyp}, we have $v_{\p}(L_{\p}^\ast(\phi;1))=c_{\p}(\phi;u_\phi^*(1))$.
\end{enumerate}
\end{thm}

\begin{proof}
Let $k$ denote as before the vanishing order of $u_\phi(T)$ at $T = 1$.
According to the class formula \eqref{classformulaT} we have:
\begin{align}\label{eq:formula-in-T}
L_{\p}(\phi;T)=(T-1)^k\cdot \p^{-1} \cdot \logT_{\phi,\p}\left(\phiT_{|\phiT(\FF_{\p})|}(u_{\phi}^*(T))\right).
\end{align}
Evaluating at $T=1$ and using that $A$ has no $\phi$-torsion, it follows from Lemma~\ref{lem:torsion}
that $\log_{\phi,\p}(\phi_{|\phi(\FF_{\p})|}(u_{\phi}^*(1)))$ is nonzero.
Hence, the vanishing order at $T = 1$ of $L_{\p}(\phi;T)$ is $k$ and the first statement is proved.

Evaluating \eqref{eq:formula-in-T} at $T=1$ results in
\[
L_{\p}^*(\phi;1) = \p^{-1}\cdot\log_{\phi,\p}(\phi_{|\phi(\FF_\p)|}(u_{\phi}^*(1)))
\]
which, after taking $\p$-adic valuations using Proposition~\ref{prop:isom-p-adic-log}, provides
\[
v_{\p}(L^*_{\p}(\phi;1))=v_{\p}(\phi_{|\phi(\FF_\p)|}(u_{\phi}^*(1)))-1.
\]
Proposition~\ref{prop:clé} then gives the annonced equality.
\end{proof}

\subsubsection{Special value: the general case}

For $h \in A$, we consider the twist $h^{-1} \phi h$.
A direct calculation shows that if $\phi_t = t + g_1 \tau + \cdots + g_r \tau^r$, then
\[
(h^{-1} \phi h)_t = t + g_1 h^{q-1} \tau + g_2 h^{q^2 - 1} + \cdots + g_r h^{q^r - 1} \tau^r
\]
and so, that $h^{-1} \phi h$ again defines a Drinfeld model.

Besides, we know from \cite[Lemma 3.7]{ANT} that the local factors of $h^{-1} \phi h$ are those of $\phi$ at the places not dividing $h$ and are trivial otherwise.
At the level of $L$-series, this translates to
\begin{align}
L\big(\phi; T\big) & = L\big(h^{-1}\phi h; T\big) \cdot \prod_{\q | h} \frac{\q}{|\phiT(\FF_\q)|}, \\
L_\p\big(\phi; T\big) & = L_\p\big(h^{-1}\phi h; T\big) \cdot \prod_{\substack{\q | h\\ \q \neq \p}} \frac{\q}{|\phiT(\FF_\q)|}. \label{eq:Lptwist}
\end{align}
This already has quite interesting consequences.

\begin{thm}
\label{thm:independence}
The vanishing order at $T=1$ of the $\p$-adic $L$-series $L_\p(\phi;T)$ is independent from $\p$.
\end{thm}

\begin{proof}
We consider a twisting element $h \in A$ such that $\deg h > \max_{1 \leq i \leq r} \deg g_i$.
A simple computation ensures that, for any $a \in A$, the leading coefficient (with respect to~$\tau$) of $(h^{-1} \phi h)_a$ has $t$-degree strictly larger than all its other coefficients.
It follows that $\phi(A)$ cannot have $a$-torsion. Since this holds for all $a$, we conclude that $A$ has no $\phi$-torsion.

We can then apply Theorem~\ref{thm:main} and deduce that the order of vanishing at $T=1$ of $L_\p\big(h^{-1}\phi h; T\big)$ does not depend on~$\p$.
The theorem follows after remarking that the factors $\frac{\q}{|\phiT(\FF_\q)|}$ appearing in Equation~\eqref{eq:Lptwist} do not vanish at $T = 1$.
\end{proof}

In order to get further information about the special values, it is convenient to focus on the special case where the twisting element $h$ is a power of $\p$.
In this case, we first observe that Equation~\eqref{eq:Lptwist} tells that $L_\p\big(\phi; T\big) = L_\p\big(\p^{-m}\phi \p^m; T\big)$ for all $m$.
It turns out that we have analoguous relations for Taelman units and ordic valuations.

\begin{lem}
\label{lem:uctwist}
For all positive integers $m$ and all $x \in A$, we have
\begin{enumerate}[label=(\theenumi)]
\item $u_{\p^{-m}\phi \p^m}(T) = \p^{-m} \cdot \phiT_{\p^{m-1}|\phiT(\FF_\p)|}(u_\phi(T))$,
\item $c_{\p}(\p^{-m}\phi\p^m;  x) = c_{\p}(\phi; \p^m x) - m$.
\end{enumerate}
\end{lem}

\begin{proof}
To simplify notation, we set $\psi := \p^{-m}\phi \p^m$ throughout the proof.
The first formula follows from the next computation:
\begin{align*}
u_{\psi}(T)
 & = \p^{-m}\cdot\expT_{\phi} \big(\p^m \cdot L(\psi; T)\big) \\
 & = \p^{-m}\cdot\expT_{\phi} \big(\p^{m-1} \cdot |\phiT(\FF_\p)| \cdot L({\phi};T)\big) \\
 & = \p^{-m}\cdot\phiT_{\p^{m-1}|\phiT(\FF_\p)|}(u_\phi(T)).
\end{align*}
For the second formula, we remark first that for $k \leq m$, we have $\pi_{\p^mx}(\phi;\p^k) = A$; hence $c_{\p}(\phi;\p^m x)\geq m$.
Moreover, for any $k \geq 0$, the equality $\psi_a(x) = \p^{-m} \cdot \phi_a(\p^m x)$ (which is correct for all $a \in A$) shows that $\pi_x(\psi; \p^k) =\pi_{\p^m x}(\phi;\p^{m+k})$.
The conclusion follows.
\end{proof}

\begin{thm}
\label{thm:vpLp}
Let $\p$ be a place satisyfing \ref{Hyp}.
For a positive integer $m$, we let $Q_m(T)$ be the polynomial defined by
\[
\phiT_{\p^{m-1} |\phiT(\FF_\p)|} \big(u_\phi(T)\big)=(T-1)^{k_m} \cdot Q_m(T), \quad Q_m(1) \neq 0.
\]
Then, for $m$ large enough, we have $v_{\p}\big(L_{\p}^*(\phi;1)\big)=c_{\p}\big(\phi;Q_{m}(1)\big)-m$.
\end{thm}

\begin{proof}
We set $\psi := \p^{-m} \phi \p^m$. Lemma~\ref{lem:uctwist} gives the relation
$$u_{\psi}(T) = \p^{-m} \cdot (T-1)^{k_m} \cdot Q_m(T)$$
from which we deduce that $k_m$ is the order of vanishing of $u_\psi(T)$ at $T = 1$ and that $u_\psi^*(T) = \p^{-m} \cdot Q_m(T)$.
On the other hand, we know from the first part of the proof of Theorem~\ref{thm:independence} that $\psi$ has no torsion when $m$ is large enough.
We can then apply Theorem~\ref{thm:main} and get
$$v_{\p}\big(L_{\p}^*(\psi,1)\big) = c_{\p}\big(\psi; \p^{-m} Q_m(1)\big).$$
Now we conclude by observing that the left hand side in the above equality is equal to $v_{\p}(L_{\p}^*(\phi;1))$ thanks to Equation~\eqref{eq:Lptwist}
while the right hand side is $c_{\p}\big(\phi; Q_m(1)\big) - m$ by Lemma~\ref{lem:uctwist}.
\end{proof}

\subsection{The smallness condition}\label{sec:small}

Previously we have proved that the $\p$-adic valuation of the special value of the $\p$-adic $L$-series is related to the ordic valuation of the special element $u_\phi(1)$ or one of its analogue.
In this last subsection, we compute this element $u_\phi(1)$ under some additional hypothesis on $\phi$; this will eventually gives a complete proof of Theorem~\ref{thm:smallWieferich} of the introduction.

We start by recalling the definition of smallness.

\begin{deftn}
\label{def:small}
Write $\phi_t=t+g_1\tau+\ldots+g_r\tau^r$.
We say that $\phi$ is \emph{small} (resp. \emph{very small}) if $\deg_{t} g_i\leq q^i$ (resp. $\deg_{t} g_i < q^i$) for all $i\in \{1,\ldots,r\}$.
\end{deftn}

\begin{prop}
\label{prop:small}
We assume that $\phi$ is small and, for $i \in \{1, \ldots, r\}$, we let $\alpha_i$ denote the coefficient of $g_i$ in front of $t^{q^i}$.
Then
$$u_\phi(T) = 1 + \sum_{i=1}^r \alpha_i T^i \in \FF[T].$$
In particular, if $\phi$ is very small, we have $u_\phi(T) = 1$.
\end{prop}

\begin{proof}
We let $(e_i)_{i\geq 0}$ denote the sequence of coefficients of $\exp_\phi$, namely
$$\exp_\phi = e_0 + e_1 \tau + e_2 \tau^2 + \cdots + e_n \tau^n + \cdots \in K\{\!\{\tau\}\!\}.$$
We also complete the sequence $(\alpha_i)$ by letting $\alpha_i = 0$ for $i = 0$ and $i > r$.

We first prove, by induction on $n$, that $v_\infty(e_n) \geq 0$ and $e_n \equiv \alpha_n \pmod{\pi}$ where $\pi := 1/t$ is the uniformizer at infinity.
The formula clearly holds for $n = 0$ since $e_0 = \alpha_0 = 1$.
Let us assume that it holds for all $i < n$. Then, from the induction formula for the coefficients of the exponential, we get
\begin{equation}
\label{eq:dn}
e_n = \frac{1}{t^{q^n}-t} \cdot \sum_{i=1}^{\min(r,n)} g_i e_{n-i}^{q^i}.
\end{equation}
For $1 \leq i \leq \min(r,n)$, we have
$$v_\infty\left(\frac{g_i e_{n-i}^{q^i}}{t^{q^n}-t}\right) = q^n - \deg g_i + q^i v_\infty(e_{n-i})\geq q^n-q^i\geq 0.$$
Thus $v_\infty(e_n) \geq 0$. 
The previous computation shows also that all the summands in the right hand side of Equation~\eqref{eq:dn} have positive valuation, expect maybe the one corresponding to $i = n$.
Therefore, reducing Equation~\eqref{eq:dn} modulo $\pi$, we get $e_n \equiv \alpha_n \pmod{\pi}$ as claimed.

We now write
\begin{equation}
\label{eq:uphi}
u_\phi(T)=\expT_{\phi}\big(L(\phi;T)\big)=\sum_{n\geq 0} e_n T^n \cdot \tau^n\big(L(\phi;T)\big).
\end{equation}
We know from~\cite[Lemma~2.2]{AT} that $v_\infty\big(L(\phi;T)\big) = 0$ and $L(\phi;T) \equiv 1 \pmod \pi$; hence reducing Equation~\eqref{eq:uphi} modulo $\pi$ yields
$$u_\phi(T) \equiv \sum_{n\geq 0} \alpha_n T^n \pmod \pi.$$
Finally, given that $\expT_{\phi} \big(L(\phi;T)\big)\in A[T]$ according to \cite[Proposition 3.2]{AT}, we obtain the result.
\end{proof}

\begin{rem}
When $\phi$ is very small, one can prove in addition that $L(\phi;T)=\logT_{\phi}(1)$.
Indeed, according to \cite[Proposition 3.2]{AT}, we have $\logT_{\phi}(1)=a\cdot L({\phi};T)$ for a nonzero $a\in A[T]$.
By taking the exponential, we obtain $1=\phiT_a(u_{\phi}(T))$ and writing $\phiT_a=\sum_{i=0}^s a_i T^i\tau^i$, we find
$$0 = \deg_T \phiT_a(u_\phi(T))=s+\deg_T u_\phi(T).$$
Thus $s=0$, and so $a\in \FF^\times$. Looking at the constant coefficient, we finally derive $a = 1$ and so $\logT_{\phi}(1) = L({\phi};T)$ as claimed.
\end{rem}

We now have all the ingredients to complete the proof of Theorem~\ref{thm:smallWieferich}.

Let us first assume that $\phi$ is very small. Then $u_\phi(T) = 1$ by Proposition~\ref{prop:small} and so $u_\phi(1) = 1$.
In particular, $u_\phi(T)$ does not vanish at $T = 1$. The first part of Theorem~\ref{thm:main} then ensures that $L_\p(\phi; 1)$ does not vanish at $T = 1$ as well, \emph{i.e.} $L_{\p}(\phi;1) = L_{\p}^*(\phi;1)$.
The second part of Theorem~\ref{thm:main} now gives the desired equality: $v_{\p}(L_{\p}(\phi;1)) = c_{\p}(\phi;1)$ for all places $\p$ satisfying the hypothesis~\ref{Hyp}.

We now assume that $\phi$ is small and has no torsion point.
Applying again Proposition~\ref{prop:small}, we find that $u_\phi(1)$ lies in $\FF[T]$, from what it follows that $u_\phi^*(1) \in \FF^\times$.
From Remark~\ref{rem:ind-Fx}, we deduce that $c_\p(\phi; u_\phi^*(1)) = c_\p(\phi; 1)$ and
Theorem~\ref{thm:main} shows finally that $v_{\p}(L_{\p}^*(\phi;1)) = c_{\p}(\phi;1)$ for all places $\p$ satisfying the hypothesis~\ref{Hyp}.

\section{Statistics on Wieferich primes}\label{sec:statistics}

In this section, we adopt a probabilistic viewpoint. Our
objective is to give credit to the naive expectation that
a place $\p$ is Wieferich with probability $q^{-\deg\p}$,
supporting eventually the fact that a given Drinfeld model
admits an infinite number of Wieferich places; indeed, the
number of places of degree $d$ is roughly $q^d/d$ and
$$\sum_{d=1}^\infty \frac{q^d} d \cdot q^{-d} =
  \sum_{d=1}^\infty \frac 1 d = +\infty.$$
It is unclear to us if this heuristic is reasonable for a
single Drinfeld model. However, in what follows, we shall
prove that it is somehow valid when we average over larger
universes.

Precisely, we fix a positive integer $r$ and we let $\Omega_r$
denote the set of all small Drinfeld models $\phi : A \to
A\{\tau\}$ of rank at most $r$. It is a finite set, and we
equip it with the uniform distribution. Given in addition
a place $\p$ of $A$, we consider the Bernoulli variable 
$$\begin{array}{rcl}
W_{r,\p} : \quad \Omega_r & \longrightarrow & \{0, 1\} 
\smallskip \\
\phi & \mapsto & 
   0 \quad \text{if $\p$ is $\phi$-Wieferich in base $1$} \\
&& 1 \quad \text{otherwise}.
\end{array}$$
Our objective is to study those random variables and their
relationships.

\begin{rem}
To define our universe, we retained the property of being small 
because it looks quite meaningful regarding our purpose after 
the results of Section~\ref{sec:small}. However, most of the 
results we shall prove in this section are not strongly
dependant on this choice, and will continue to hold true
for many other families of universes, as soon as they 
eventually allow for arbitrary large ranks and arbitrary
large degrees in the coefficients of the Drinfeld models.
\end{rem}

\subsection{The case of places of degree $1$}

To start with, we consider the case where $\p$ is a place
of degree~$1$. In this situation, we have a simple criterion
for recognizing when a place is Wieferich. Before stating it,
we recall that if $F = f_0 + f_1 \tau + \cdots + f_n \tau^n
\in A\{\tau\}$ is a polynomial in $\tau$ and if $x \in A$, we 
write
$$F(x) = f_0 x + f_1 x^q + \cdots  + f_n x^{q^n}.$$
It is an element of $A$, \emph{i.e.} a polynomial in $t$ 
over $\Fq$, and we will denote by $\frac {df(x)}{dt}$ its
derivative with respect to~$t$. A direct computation shows
that
$$\frac{dF(x)}{dt} = 
  f'_0 x + f_0 x' + f'_1 x^q + \cdots  + f'_n x^{q^n}.$$

\begin{prop}
\label{prop:wieferichdeg1}
Let $\phi : A \to A\{\tau\}$ be a Drinfeld model. Let also
$\alpha \in \Fq$ and $x \in A$, $x \neq 0$.
Then the place $\p = t - \alpha$ is $\phi$-Wieferich in base
$x$ if and only if $\frac {d\phi_t(x)}{dt} (\alpha) = 0$.
\end{prop}

\begin{proof}
Set $f = \phi_t(x)$ and write $\beta = f(\alpha)$.
We observe that
$$\phi_{t-\beta}(x) = f - \beta 
  \equiv \beta - \beta = 0 \pmod{\p}$$
which shows that $\pi_x(\phi; \p)$ is the principal ideal generated 
by $t{-}\beta$. It follows that $\p$ is $\phi$-Wieferich in base $x$ 
if and only if $\p^2$ divides $f - \beta$. Besides, using 
Taylor expansion, we get the congruence
$f \equiv \beta + \p f'(\alpha) \pmod{\p^2}$,
from what we conclude that $\p$ is $\phi$-Wieferich in base $x$
if and only if $f'(\alpha)$ vanishes.
\end{proof}

\begin{cor}
\label{cor:proba:deg1}
Let $r$ be a positive integer.
\begin{enumerate}[label=(\roman{enumi})]
\item 
\label{item:proba:parameter}
For any place $\p$ of degree $1$,
the Bernoulli variable $W_{r,\p}$ takes the value $1$ with
probability $q^{-1}$.
\item
\label{item:proba:independence}
The variables $W_{r,\p}$ are mutually independent
when $\p$ runs over the set of places of degree~$1$.
\end{enumerate}
\end{cor}

\begin{proof}
Proposition~\ref{prop:wieferichdeg1} tells us that the places of
degree $1$ which are $\phi$-Wieferich in base $x$ and in one-to-one
correspondence with the roots in $\Fq$ of the polynomial 
$\frac {d\phi_t(x)}{dt}$. When $\phi$ varies in $\Omega_r$, the
polynomial $\phi_t(x)$ is uniformly distributed in the space 
$A_{\leq q^r}$ consisting of polynomials over $\Fq$ of degree at
most $q^r$.
We now conclude by noticing that the map
$$\begin{array}{rcl}
A_{\leq q^r} & \longrightarrow & \Fq^{\Fq} \\
f & \mapsto & \big(f(\alpha)\big)_{\alpha \in \Fq}
\end{array}$$
is $\Fq$-linear and surjective, thanks to Lagrange interpolation.
\end{proof}

\subsection{The probability of being a Wieferich prime}
\label{ssec:proba}

We now aim at extending Corollary~\ref{cor:proba:deg1} to
places of higher degrees. We start by determining the parameters
of the Bernoulli variables $W_{r,\p}$ at least when the rank is
large compared to the degree of $\p$.

\begin{thm}
\label{th:proba:parameter}
Let $r$ and $d$ be two positive integers with $r \geq d + \log_q(2d)$.
For any place $\p$ of degree $d$, the Bernoulli variable $W_{r,\p}$ 
takes the value $1$ (\emph{i.e.} $\p$ is Wieferich in base $1$)
with probability $q^{-d}$.
\end{thm}

The rest of this subsection is devoted to the proof of 
Theorem~\ref{th:proba:parameter}.
Unfortunately, it seems difficult to follow the same strategy 
we used for places of degree $1$. Indeed, although Wieferich 
places of higher degrees $d$ can certainly be characterized
by the vanishing of some polynomial, it looks difficult to study
its distribution when the underlying Drinfeld model $\phi$
varies.

Instead, we will use another characterization of Wieferich
places, that we explain now.
Let $\p$ be a place of $A$ of degree $d$ and let $\Fp := A/\p$
denote the corresponding residual field as before. We first partition 
the universe $\Omega_r$ according to the reduction modulo~$\p$:
given a Drinfeld module $\phibar : A \to \Fp\{\tau\}$, we let 
$\Omega_r(\phibar)$ be the subset of $\Omega_r$ consisting of
Drinfeld models $\phi$ which reduces to $\phibar$ modulo $\p$.
We are going to prove that the proportion of Drinfeld models
admitting $\p$ as a Wieferich place inside each nonempty
$\Omega_r(\phibar)$ is $q^{-d}$; this will be enough to
conclude.

From now on, we fix $\phibar$ as above, together with a 
lifting $\phi \in \Omega_r(\phibar)$. If $\psi$ is a second
Drinfeld model in $\Omega_r(\phibar)$, we have an equality
of the form $\psi_t = \phi_t + \p f$ with $f \in A\{\tau\}$.
Moreover, $f$ takes the form
$f = f_1 \tau + f_2 \tau^2 + \cdots + f_r \tau^r$
with $\deg f_i \leq q^i - d$. 
We will denote by $\Omega'_r$ this set in which $f$ varies.

\begin{lem}
\label{lem:proba:differential}
Keeping the previous notation, we have
$$\psi_{t^i} \equiv \phi_{t^i} + 
  \p \cdot \sum_{j=0}^{i-1} t^j f \phi_t^{i-j-1} \pmod{\p^2}.$$
\end{lem}

\begin{proof}
We proceed by induction on $i$. For $i = 1$, the equality
we have to prove is just $\psi_t = \phi_t + \p f$, which is
true by definition of $f$. We now assume that the equality
holds for~$i$. We compute
\begin{align*}
\psi_{t^{i+1}} 
& = \psi_{t^i} \psi_t
  = \left(\phi_{t^i} + \p \cdot \sum_{j=0}^{i-1} t^j f \phi_t^{i-j-1}\right)
    \cdot \big(\phi_t + \p f\big) \\
& \equiv \phi_{t^{i+1}} + \p \cdot \sum_{j=0}^{i-1} t^j f \phi_t^{i-j}
  + \phi_t^i \p f \pmod{\p^2}.
\end{align*}
To handle the last summand, we write
$\phi_t^i = \phi_{t^i} = t^i + h \tau$ with
$h \in A\{\tau\}$. From this, we derive
$$\phi_t^i \p f = t^i \p f + h \tau \p f = t^i \p f + h \p^q \tau f
\equiv t^i \p f \pmod{\p^2}.$$
Injecting finally this is the first equality, we obtain the
announced formula.
\end{proof}

We notice now that the ideal $\pi_1(\psi; \p)$ does not depend
on $\psi \in \Omega_r(\phibar)$, but only on $\phibar$. Let
us simply write
$a = a_0 + a_1 t + \cdots + a_d t^d$ (with $a_i \in \FF$)
for the monic generator of this ideal.

\begin{prop}
\label{prop:proba:characterization}
We keep the previous notation and let further $\xi$ (resp.
$\mu_j$) be the image of $t$ (resp. of $\phi_{t^j}(1)$) in $\Fp$.
Then the place $\p$ is $\psi$-Wieferich in base $1$ if and only
if
$$\sum_{i=1}^d \sum_{j=0}^{i-1} a_i \xi^j f(\mu_{i-j-1}) \equiv 
  -\frac{\phi_a(1)}{\p} \pmod \p.$$
\end{prop}

\begin{proof}
It follows from Lemma~\ref{lem:proba:differential} that
$$\psi_a(1) \equiv 
\phi_a(1) + \p \sum_{i=1}^d \sum_{j=0}^{i-1} 
  a_i \cdot t^j \cdot (f\phi_t^{i-j-1})(1)
  \pmod{\p^2}.$$
Therefore the condition of the proposition is equivalent to
the vanishing of $\psi_a(1)$ modulo $\p^2$ which is, by 
definition, also equivalent to the fact that $\p$ is 
$\psi$-Wieferich in base $1$.
\end{proof}

The main insight of Proposition~\ref{prop:proba:characterization}
is that it provides a \emph{linear} characterization of the
property of being Wieferich. To fully exploit this fact, we
introduce the mapping
$$\begin{array}{rcl}
L_r : \quad
\Omega'_r & \longrightarrow & \Fp \\
f & \mapsto & 
  \displaystyle \sum_{i=1}^d \sum_{j=0}^{i-1} a_i \xi^j f(\mu_{i-j-1})
\end{array}$$
It is $\FF$-linear. Besides, Proposition~\ref{prop:proba:characterization}
ensures that the Drinfeld models $\psi \in \Omega_r(\phibar)$ admitting
$\p$ as a Wieferich place in base $1$ are in one-to-one correspondence
with the inverse image by $L_r$ of the element
$-\frac{\phi_a(1)}{\p} \in \Fp$. Proving that this event holds with
probability $q^{-d} = \frac 1 {\text{Card}(\Fp)}$ then amounts to
proving that $L_r$ is surjective.

\begin{lem}
\label{lem:proba:nonzero}
Let $j \geq 1$. At least one of the elements
$L_r(\tau^j), L_r(\tau^{j+1}), \ldots, 
L_r(\tau^{j+d-1})$ does not vanish.
\end{lem}

\begin{proof}
We assume by contradiction that
$L_r(\tau^j) = \cdots = L_r(\tau^{j+d-1}) = 0$.
By definition, we have
$$L_r(\tau^k) = \sum_{i=1}^d \sum_{j=0}^{i-1} a_i \xi^j \mu_{i-j-1}^{q^k}.$$
By assumption, the latter vanishes for $k$ varying between $j$ and 
$j{+}d{-}1$. Nonetheless, given that $\mu_{i-j-1} \in \Fp$, we have
$\mu_{i-j-1}^{q^d} = \mu_{i-j-1}$, from what we conclude that the
vanishing holds for all $k \in \ZZ$.

Using algebraic transformations, we are going to prove that this
implies other vanishings. Precisely, for $n, s \geq 0$, we set
$$\xi_{n,s} := \sum_{\substack{e_0, \ldots, e_{n-1} \geq 0 \\ 
                              e_0 + \cdots + e_{n-1} = s}} 
  \xi^{e_0 + q e_1 + \cdots + q^{n-1} e_{n-1}}.$$
We observe that $\xi_{n,0} = 1$ for all $n$, and that $\xi_{1,s} =
\xi^s$ for all $s$. Thus our assumption reads
$$\forall k \in \ZZ, \qquad
  \sum_{i=1}^d \sum_{j=0}^{i-1} 
  a_i \cdot \xi_{1,j} \cdot \mu_{i-j-1}^{q^k} = 0.$$
We will prove by induction on $n$ that
\begin{equation}
\label{eq:proba:HR}
\forall k \in \ZZ, \qquad
  \sum_{i=n}^d \sum_{j=0}^{i-n} 
  a_i \cdot \xi_{n,j} \cdot \mu_{i-j-n}^{q^k} = 0.
\end{equation}
for all $n \in \{1, \ldots, d\}$. Starting from the induction hypothesis
for some $n < d$ (with $k$ replaced by $k{-}1$) and raising it to the $q$-th
power, we get
$$\sum_{i=n}^d \sum_{j=0}^{i-n} 
  a_i \cdot \xi_{n,j}^q \cdot \mu_{i-j-n}^{q^k} = 0.$$
It follows that
\begin{equation}
\label{eq:proba:HR:step}
\sum_{i=n}^d \sum_{j=1}^{i-n} 
  a_i \cdot (\xi_{n,j}^q - \xi_{n,j}) \cdot \mu_{i-j-n}^{q^k} = 0.
\end{equation}
Note that the sum over $j$ could safely starts at $1$ because the
terms corresponding to $j=0$ all vanish. We now claim that the
following identity holds:
\begin{equation}
\label{eq:proba:relxi}
\xi_{n,j}^q - \xi_{n,j} = \xi_{n+1,j-1} \cdot (\xi^{q^n} - \xi).
\end{equation}
Indeed, we first observe that
$$\xi_{n,j}^q - \xi_{n,j}
  = \sum_{\substack{e_1, \ldots, e_n \geq 0 \\ 
                    e_1 + \cdots + e_n = j}} 
    \xi^{q e_1 + q^2 e_2 + \cdots + q^n e_n} 
  - \sum_{\substack{e_0, \ldots, e_{n-1} \geq 0 \\ 
                    e_0 + \cdots + e_{n-1} = j}} 
    \xi^{e_0 + q e_1 + \cdots + q^{n-1} e_{n-1}}.$$
The terms with $e_n = 0$ in the first sum cancel with the terms with
$e_0 = 0$ is the second sum. Therefore, we do not change the value of
the difference if we remove those terms. Performing in addition the
changes of variables $e_0 \mapsto e_0 - 1$ et $e_n \mapsto e_n - 1$,
we end up with
$$\begin{array}{l}
\xi_{n,j}^q - \xi_{n,j} = \medskip \\
\hspace{3ex} \displaystyle
    \sum_{\substack{e_1, \ldots, e_n \geq 0 \\ 
                    e_1 + \cdots + e_n = j-1}} 
    \xi^{q e_1 + q^2 e_2 + \cdots + q^n (e_n + 1)} 
  - \sum_{\substack{e_0, \ldots, e_{n-1} \geq 0 \\ 
                    e_0 + \cdots + e_{n-1} = j-1}} 
    \xi^{(e_0+1) + q e_1 + \cdots + q^{n-1} e_{n-1}}.
\end{array}$$
Similarly, we compute
$$\begin{array}{l}
\xi_{n+1,j-1} \cdot (\xi^{q^n} - \xi) =
\xi_{n+1,j-1} \xi^{q^n} - \xi_{n+1,j-1} \xi = \medskip \\
\hspace{3ex} \displaystyle
    \sum_{\substack{e_0, \ldots, e_n \geq 0 \\ 
                    e_0 + \cdots + e_n = j-1}} 
    \xi^{e_0 + q e_1 + q^2 e_2 + \cdots + q^n (e_n + 1)} 
  - \sum_{\substack{e_0, \ldots, e_n \geq 0 \\ 
                    e_0 + \cdots + e_n = j-1}} 
    \xi^{(e_0+1) + q e_1 + \cdots + q^n e_n}.
\end{array}$$
Again the terms in the first sum with $e_0 > 0$ cancel with the terms
in the second sum with $e_n > 0$, leading to
$$\begin{array}{l}
\xi_{n+1,j-1} \cdot (\xi^{q^n} - \xi) = \medskip \\
\hspace{3ex} \displaystyle
    \sum_{\substack{e_1, \ldots, e_n \geq 0 \\ 
                    e_1 + \cdots + e_n = j-1}} 
    \xi^{q e_1 + q^2 e_2 + \cdots + q^n (e_n + 1)} 
  - \sum_{\substack{e_0, \ldots, e_{n-1} \geq 0 \\ 
                    e_0 + \cdots + e_{n-1} = j-1}} 
    \xi^{(e_0+1) + q e_1 + \cdots + q^{n-1} e_{n-1}}.
\end{array}$$
The equality~\eqref{eq:proba:relxi} follows. Injecting it in
Equation~\eqref{eq:proba:HR:step}, we obtain
$$ (\xi^{q^n} - \xi) \cdot \sum_{i=n}^d \sum_{j=1}^{i-n} 
  a_i \cdot \xi_{n+1,j-1} \cdot \mu_{i-j-n}^{q^k} = 0.$$
The prefactor $\xi^{q^n}{-}\xi$ does not vanish because $\xi$
generates $\Fp$ over $\FF$ and $n < [\Fp:\FF] = d$. We can then safely
delete it. Performing finally the change of variables $j \mapsto j+1$,
we find Equation~\eqref{eq:proba:HR} for~$n{+}1$ and the induction goes.

We conclude by considering the system of equations~\eqref{eq:proba:HR}
as a linear system on the $a_i$.
For a fixed $n$ (and $k = 0$), Equation~\eqref{eq:proba:HR} is of the form
$$a_n + \star\: a_{n+1} + \star\: a_{n+2} + \cdots + \star\: a_d = 0$$
where the symbols $\star$ hide some coefficients in $\Fp$. It follows
that $a_1 = \cdots = a_d = 0$. The polynomial 
$a = \pi_1(\phibar; \p)$ would then be constant, which is not possible
and proves the lemma.
\end{proof}

From this point, it is easy to prove the surjectivity of $L_r$ (and
then Theorem~\ref{th:proba:parameter}). Let $j$ be the smallest 
positive integer such that $q^j \geq 2d - 1$; one has $j \leq 1 + 
\log_q(2d)$.
On the one hand, by Lemma~\ref{lem:proba:nonzero}, there exists
$k \in \{0, \ldots,  d{-}1\}$ such that $L_r(\tau^{j+k}) \neq 0$.
On the other hand, thanks to our assumptions on $j$ and $r$, the 
Ore
polynomial $h(t){\cdot}\tau^{j+k}$ lies in $\Omega'_r$ for all
$h(t) \in A$ of degree at most $d{-}1$. Thus the image of $L_r$
contains all the elements of the form $h(\xi) L_r(\tau^{j+k})$, 
that are all the elements of $\Fp$. The surjectivity follows.

\subsection{Independence results}

Now we have determined the law of $W_{r,\p}$, we focus on their
relationships, looking for statements in line with
Corollary~\ref{cor:proba:deg1}.\ref{item:proba:independence}.
The theorem we shall prove is the following.

\begin{thm}
\label{th:proba:independence}
Let $\p_1, \ldots, \p_n$ be $n$ pairwise distinct places of $A$
and set $d_m := \deg \p_i$ for all $m \in \{1, \ldots, n\}$.
Then, the random variables $W_{r,\p_1}, \ldots, W_{r,\p_n}$ are
mutually independent as soon as
$$r \geq \max(d_1, \ldots, d_n) + \log_q\big(2 (d_1 + \cdots + d_n)\big).$$
\end{thm}

\begin{proof}
We follow the same method as in \S \ref{ssec:proba}.
For each $m \in \{1, \ldots, n\}$, we write $\Fpm := A/\p_m$.
Given a family of Drinfeld modules
$\phibar_m : A \to \Fpm\{\tau\}$ ($1 \leq m \leq n$), we
consider the set $\Omega_r(\phibar_1, \ldots, \phibar_n) \subset
\Omega_r$ consisting of Drinfeld models which reduces to $\phibar_m$
modulo $\p_m$ for all $m$.
We assume that $\Omega_r(\phibar_1, \ldots, \phibar_n)$ is not
empty and we fix a Drinfeld model $\phi : A \to A\{\tau\}$ in it.
Any other Drinfeld model 
$\psi \in \Omega_r(\phibar_1, \ldots, \phibar_n)$ is such that
$\psi_t = \phi_t + \p_1 \cdots \p_n f$ with $f$ varying in the 
set
$$\Omega'_r = \big\{\,
  f_0 + f_1 \tau + \cdots + f_r \tau^r \in A\{\tau\}
  \,\,\, \text{s.t.} \,\,
  \deg f_i \leq q^i - (d_1 + \cdots + d_n) \text{ for all } i\,\big\}.$$
For a fixed index $m$, let $\xi_m$ and $\mu_{m,j}$ ($0 \leq j \leq r$) 
be the images in $\Fpm$ of $t$ and $\phi_{t^j}(1)$ respectively.
Let also $a_m = a_{m,0} + a_{m,1} t + \cdots + a_{m,d} t^d \in A$
be the monic generator of $\pi_1(\psi; \p_m)$ and set
$$u_m := 
  \prod_{\substack{1 \leq m' \leq n \\ m' \neq m}} \p_{m'}(\xi_m)
  \in \Fpm.$$
Since $\p_{m'}$ is coprime with $\p_m$ for all $m' \neq m$, 
we have $u_m \neq 0$.
Repeating the proof of Proposition~\ref{prop:proba:characterization}, 
we find that the place $\p_m$ is $\psi$-Wieferich in base $1$ if and 
only if
$$\sum_{i=1}^d \sum_{j=0}^{i-1} a_{m,i} \xi_m^j f(\mu_{m,i-j-1}) \equiv
  -u_m \cdot \frac{\phi_a(1)}{\p} \pmod {\p_m}.$$
We now consider the $\FF$-linear map
$$\begin{array}{rcl}
L_{r,m} : \quad
\Omega'_r & \longrightarrow & \Fpm \\
f & \mapsto & 
  \displaystyle \sum_{i=1}^d \sum_{j=0}^{i-1} a_i \xi_m^j f(\mu_{m,i-j-1})
\end{array}$$
and
$L_r : \Omega'_r \to \FF_{\p_1} \times \cdots \times \FF_{\p_n}$, 
$f \mapsto \big(L_{r,m}(f)\big)_{1 \leq m \leq n}$.
As in \S \ref{ssec:proba} (see Lemma~\ref{lem:proba:nonzero} and
the discussion thereafter), we prove that $L_{r,m}$ is surjective.
More precisely, fixing $j \geq \log_q(2(d_1 {+} \cdots {+} d_n))$,
any element $\alpha_m \in \Fpm$ has a preimage of the form
$$f_m = f_{m,j} \tau^j + \cdots + f_{m,j+d_m-1} \tau^{j+d_m-1}
\in A\{\tau\}.$$
This property implies the surjectivity of $L_r$ as follows.
We pick $\alpha_1 \in \FF_{\p_1}, \ldots, \alpha_n \in \FF_{\p_n}$ 
and, for each $m$, we choose a preimage $f_m$ of $\alpha_m$ of the
above form. Setting $d := \max(d_1, \ldots, d_n)$ and applying the 
chinese remainder theorem separately on each coefficient, we find
$g = g_j \tau^j + \cdots + g_{j+d-1} \tau^{j+d-1} \in A\{\tau\}$
such that $g \equiv f_m \pmod{\p_m}$ for all $m$.
Moreover, we may assume that all the $g_i$ have $t$-degree less 
than $d_1 + \cdots + d_n$ because they are defined modulo $\p_1 
\cdots \p_n$. Thanks to our assumption on $r$, we then have $g \in 
\Omega'_r$. Finally, noticing that $L_{r,m}$ depends only on the
reduction of $g$ modulo $\p_m$, we conclude that $L_{r,m}(g) = 
(\alpha_1, \ldots, \alpha_n)$, proving the surjectivity.

It follows that the probability that $L_r$ takes the
value $(\alpha_1, \ldots, \alpha_n)$ is constant, namely
$$\frac 1{\text{Card}\big(\FF_{\p_1} \times \cdots \times \FF_{\p_n}\big)}
 = q^{-(d_1 + \cdots + d_n)}.$$
This proves the independence of the random variables $L_{r,m}$
which, in turn, implies the independence of the $W_{r,\p_m}$.
\end{proof}

Specializing Theorem~\ref{th:proba:independence} to the case where 
we consider all places of a fixed degree~$d$ (resp. of degree at
most~$d$), we get the following.

\begin{cor}
\label{cor:proba:independence}
Let $r$ and $d$ be two positive integers.
\begin{enumerate}[label=(\roman{enumi})]
\item
\label{item:indep:degd}
If $r \geq 2d$, the random variables $W_{r,\p}$, for $\p$ running
over all places of degree $d$, are mutually independent.
\item
\label{item:indep:degltd}
If $r \geq 2d + 1$, the random variables $W_{r,\p}$, for $\p$ running 
over all places of degree at most $d$, are mutually independent.
\end{enumerate}
\end{cor}

\begin{proof}
Let $\mathcal P_d$ (resp.~$\mathcal P_{\leq d}$) denote the set 
of places of degree~$d$ (resp. of degree at most~$d$).
After Theorem~\ref{th:proba:independence}, it only remains to
estimate the sum of $\deg \p$ for $\p$ running in~$\mathcal P_d$ 
(resp. in~$\mathcal P_{\leq d}$).
This is quite standard and follows from the
observation that each place of degree $d$ is a divisor of the
polynomial $t^{q^d} - t$. Therefore their product also divides
$t^{q^d} - t$ and we conclude that
$\sum_{\p \in \mathcal P_d} \deg \p \leq q^d$. 
The statement~\ref{item:indep:degd} follows.

Similarly, we have
$$\sum_{\p \in \mathcal P_{\leq d}} \deg \p
\leq \sum_{\delta=1}^d q^\delta 
= \frac{q^{d+1} - q}{q-1} < q^{d+1}$$
which eventually implies~\ref{item:indep:degltd}.
\end{proof}

\subsection{Discussion and numerical simulations}

\begin{figure}
\begin{tabular}{|l|r|r|r|r|r|r|r|r|r|r|}
\cline{2-11}
\multicolumn{1}{c|}{} & \multicolumn{2}{c|}{\textbf{rank 1}} & \multicolumn{2}{c|}{\textbf{rank 2}} & \multicolumn{2}{c|}{\textbf{rank 3}} & \multicolumn{2}{c|}{\textbf{rank 4}} & \multicolumn{2}{c|}{\textbf{rank 5}} \\
\multicolumn{1}{c|}{} & all & n.t. & all & n.t. & all & n.t. & all & n.t. & all & n.t. \\
\hline
\textbf{deg 1} & $\phantom{0}1.14$ & $\phantom{0}0.80$ & $\phantom{0}1.00$ & $\phantom{0}0.94$ & $\phantom{0}0.99$ & $\phantom{0}0.99$ & $\phantom{0}0.99$ & $\phantom{0}0.99$ & $\phantom{0}1.01$ & $\phantom{0}1.01$ \\
\textbf{deg 2} & $\phantom{0}1.71$ & $\phantom{0}0.80$ & $\phantom{0}1.24$ & $\phantom{0}1.08$ & $\phantom{0}1.00$ & $\phantom{0}0.99$ & $\phantom{0}0.99$ & $\phantom{0}0.99$ & $\phantom{0}0.97$ & $\phantom{0}0.97$ \\
\textbf{deg 3} & $\phantom{0}3.43$ & $\phantom{0}1.60$ & $\phantom{0}0.87$ & $\phantom{0}0.44$ & $\phantom{0}1.13$ & $\phantom{0}1.10$ & $\phantom{0}1.06$ & $\phantom{0}1.06$ & $\phantom{0}1.00$ & $\phantom{0}1.00$ \\
\textbf{deg 4} & $\phantom{0}6.86$ & $\phantom{0}3.20$ & $\phantom{0}2.06$ & $\phantom{0}1.23$ & $\phantom{0}0.98$ & $\phantom{0}0.93$ & $\phantom{0}1.03$ & $\phantom{0}1.03$ & $\phantom{0}1.03$ & $\phantom{0}1.03$ \\
\textbf{deg 5} & $13.71$ & $\phantom{0}6.40$ & $\phantom{0}2.54$ & $\phantom{0}0.77$ & $\phantom{0}1.05$ & $\phantom{0}0.95$ & $\phantom{0}0.98$ & $\phantom{0}0.98$ & $\phantom{0}1.01$ & $\phantom{0}1.01$ \\
\textbf{deg 6} & $27.43$ & $12.80$ & $\phantom{0}5.22$ & $\phantom{0}1.70$ & $\phantom{0}1.17$ & $\phantom{0}0.96$ & $\phantom{0}1.02$ & $\phantom{0}1.02$ & $\phantom{0}0.97$ & $\phantom{0}0.97$ \\
\textbf{deg 7} & $54.86$ & $25.60$ & $\phantom{0}8.49$ & $\phantom{0}1.34$ & $\phantom{0}1.38$ & $\phantom{0}0.97$ & $\phantom{0}1.00$ & $\phantom{0}1.00$ & $\phantom{0}1.04$ & $\phantom{0}1.04$ \\
\textbf{deg 8} & \small{${>}100$} & $58.03$ & $16.41$ & $\phantom{0}2.08$ & $\phantom{0}1.86$ & $\phantom{0}1.04$ & $\phantom{0}1.03$ & $\phantom{0}1.03$ & $\phantom{0}0.98$ & $\phantom{0}0.98$ \\
\textbf{deg 9} & \small{${>}100$} & \small{${>}100$} & $29.86$ & $\phantom{0}1.02$ & $\phantom{0}2.65$ & $\phantom{0}1.01$ & $\phantom{0}1.00$ & $\phantom{0}1.00$ & $\phantom{0}1.01$ & $\phantom{0}1.01$ \\
\textbf{deg 10} & \small{${>}100$} & \small{${>}100$} & $62.10$ & $\phantom{0}4.55$ & $\phantom{0}4.28$ & $\phantom{0}1.00$ & $\phantom{0}1.03$ & $\phantom{0}1.03$ & $\phantom{0}1.06$ & $\phantom{0}1.06$ \\
\textbf{deg 11} & \small{${>}100$} & \small{${>}100$} & \small{${>}100$} & $\phantom{0}1.69$ & $\phantom{0}7.56$ & $\phantom{0}1.01$ & $\phantom{0}1.05$ & $\phantom{0}1.05$ & $\phantom{0}0.97$ & $\phantom{0}0.97$ \\
\textbf{deg 12} & \small{${>}100$} & \small{${>}100$} & \small{${>}100$} & $\phantom{0}4.70$ & $14.19$ & $\phantom{0}1.08$ & $\phantom{0}1.06$ & $\phantom{0}1.06$ & $\phantom{0}1.05$ & $\phantom{0}1.05$ \\
\textbf{deg 13} & \small{${>}100$} & \small{${>}100$} & \small{${>}100$} & $\phantom{0}0.67$ & $27.26$ & $\phantom{0}1.04$ & $\phantom{0}0.96$ & $\phantom{0}0.96$ & $\phantom{0}0.95$ & $\phantom{0}0.95$ \\
\textbf{deg 14} & \small{${>}100$} & \small{${>}100$} & \small{${>}100$} & $11.04$ & $53.48$ & $\phantom{0}1.05$ & $\phantom{0}1.02$ & $\phantom{0}1.02$ & $\phantom{0}0.94$ & $\phantom{0}0.94$ \\
\textbf{deg 15} & \small{${>}100$} & \small{${>}100$} & \small{${>}100$} & $\phantom{0}2.18$ & \small{${>}100$} & $\phantom{0}1.04$ & $\phantom{0}1.03$ & $\phantom{0}1.03$ & $\phantom{0}0.97$ & $\phantom{0}0.97$ \\
\textbf{deg 16} & \small{${>}100$} & \small{${>}100$} & \small{${>}100$} & $20.18$ & \small{${>}100$} & $\phantom{0}1.16$ & $\phantom{0}1.00$ & $\phantom{0}1.00$ & $\phantom{0}1.11$ & $\phantom{0}1.11$ \\
\textbf{deg 17} & \small{${>}100$} & \small{${>}100$} & \small{${>}100$} & $\phantom{0}1.31$ & \small{${>}100$} & $\phantom{0}1.04$ & $\phantom{0}0.98$ & $\phantom{0}0.98$ & $\phantom{0}1.06$ & $\phantom{0}1.06$ \\
\textbf{deg 18} & \small{${>}100$} & \small{${>}100$} & \small{${>}100$} & $37.62$ & \small{${>}100$} & $\phantom{0}1.50$ & $\phantom{0}1.14$ & $\phantom{0}1.14$ & $\phantom{0}1.09$ & $\phantom{0}1.09$ \\
\textbf{deg 19} & \small{${>}100$} & \small{${>}100$} & \small{${>}100$} & $\phantom{0}0.81$ & \small{${>}100$} & $\phantom{0}0.94$ & $\phantom{0}1.05$ & $\phantom{0}1.05$ & $\phantom{0}0.97$ & $\phantom{0}0.97$ \\
\hline
\end{tabular}

\caption{Statistics of Wieferich places with $q = 2$}
\label{fig:stats:q2}
\end{figure}

\begin{figure}
\begin{tabular}{|l|r|r|r|r|r|r|r|r|r|r|}
\cline{2-11}
\multicolumn{1}{c|}{} & \multicolumn{2}{c|}{\textbf{rank 1}} & \multicolumn{2}{c|}{\textbf{rank 2}} & \multicolumn{2}{c|}{\textbf{rank 3}} & \multicolumn{2}{c|}{\textbf{rank 4}} & \multicolumn{2}{c|}{\textbf{rank 5}} \\
\multicolumn{1}{c|}{} & all & n.t. & all & n.t. & all & n.t. & all & n.t. & all & n.t. \\
\hline
\textbf{deg 1} & $\phantom{0}1.01$ & $\phantom{0}0.94$ & $\phantom{0}1.00$ & $\phantom{0}1.00$ & $\phantom{0}1.00$ & $\phantom{0}1.00$ & $\phantom{0}1.00$ & $\phantom{0}1.00$ & $\phantom{0}1.00$ & $\phantom{0}1.00$ \\
\textbf{deg 2} & $\phantom{0}0.90$ & $\phantom{0}0.58$ & $\phantom{0}1.00$ & $\phantom{0}1.00$ & $\phantom{0}1.01$ & $\phantom{0}1.01$ & $\phantom{0}1.02$ & $\phantom{0}1.02$ & $\phantom{0}1.03$ & $\phantom{0}1.03$ \\
\textbf{deg 3} & $\phantom{0}2.19$ & $\phantom{0}1.23$ & $\phantom{0}0.99$ & $\phantom{0}0.99$ & $\phantom{0}0.98$ & $\phantom{0}0.98$ & $\phantom{0}1.01$ & $\phantom{0}1.01$ & $\phantom{0}0.98$ & $\phantom{0}0.98$ \\
\textbf{deg 4} & $\phantom{0}4.56$ & $\phantom{0}1.58$ & $\phantom{0}1.00$ & $\phantom{0}1.00$ & $\phantom{0}0.98$ & $\phantom{0}0.98$ & $\phantom{0}0.98$ & $\phantom{0}0.98$ & $\phantom{0}1.01$ & $\phantom{0}1.01$ \\
\textbf{deg 5} & $\phantom{0}9.11$ & $\phantom{0}0.00$ & $\phantom{0}1.00$ & $\phantom{0}1.00$ & $\phantom{0}1.04$ & $\phantom{0}1.04$ & $\phantom{0}0.98$ & $\phantom{0}0.98$ & $\phantom{0}1.00$ & $\phantom{0}1.00$ \\
\textbf{deg 6} & $28.59$ & $\phantom{0}1.31$ & $\phantom{0}1.12$ & $\phantom{0}1.12$ & $\phantom{0}1.04$ & $\phantom{0}1.04$ & $\phantom{0}1.00$ & $\phantom{0}1.00$ & $\phantom{0}0.98$ & $\phantom{0}0.98$ \\
\textbf{deg 7} & $82.01$ & $\phantom{0}0.00$ & $\phantom{0}1.00$ & $\phantom{0}1.00$ & $\phantom{0}1.04$ & $\phantom{0}1.04$ & $\phantom{0}0.98$ & $\phantom{0}0.98$ & $\phantom{0}1.01$ & $\phantom{0}1.01$ \\
\textbf{deg 8} & \small{${>}100$} & $\phantom{0}0.00$ & $\phantom{0}1.44$ & $\phantom{0}1.44$ & $\phantom{0}0.99$ & $\phantom{0}0.99$ & $\phantom{0}0.99$ & $\phantom{0}0.99$ & $\phantom{0}1.03$ & $\phantom{0}1.03$ \\
\textbf{deg 9} & \small{${>}100$} & $\phantom{0}1.64$ & $\phantom{0}0.98$ & $\phantom{0}0.98$ & $\phantom{0}1.00$ & $\phantom{0}1.00$ & $\phantom{0}1.01$ & $\phantom{0}1.01$ & $\phantom{0}1.01$ & $\phantom{0}1.01$ \\
\textbf{deg 10} & \small{${>}100$} & $\phantom{0}5.09$ & $\phantom{0}2.17$ & $\phantom{0}2.17$ & $\phantom{0}1.00$ & $\phantom{0}1.00$ & $\phantom{0}1.06$ & $\phantom{0}1.06$ & $\phantom{0}1.01$ & $\phantom{0}1.01$ \\
\textbf{deg 11} & \small{${>}100$} & $\phantom{0}0.00$ & $\phantom{0}0.96$ & $\phantom{0}0.96$ & $\phantom{0}0.94$ & $\phantom{0}0.94$ & $\phantom{0}0.96$ & $\phantom{0}0.96$ & $\phantom{0}1.03$ & $\phantom{0}1.03$ \\
\textbf{deg 12} & \small{${>}100$} & $\phantom{0}4.06$ & $\phantom{0}4.42$ & $\phantom{0}4.42$ & $\phantom{0}0.97$ & $\phantom{0}0.97$ & $\phantom{0}1.02$ & $\phantom{0}1.02$ & $\phantom{0}0.99$ & $\phantom{0}0.99$ \\
\hline
\end{tabular}

\caption{Statistics of Wieferich places with $q = 3$}
\label{fig:stats:q3}
\end{figure}

\begin{figure}
\begin{tabular}{|l|r|r|r|r|r|r|r|r|r|r|}
\cline{2-11}
\multicolumn{1}{c|}{} & \multicolumn{2}{c|}{\textbf{rank 1}} & \multicolumn{2}{c|}{\textbf{rank 2}} & \multicolumn{2}{c|}{\textbf{rank 3}} & \multicolumn{2}{c|}{\textbf{rank 4}} & \multicolumn{2}{c|}{\textbf{rank 5}} \\
\multicolumn{1}{c|}{} & all & n.t. & all & n.t. & all & n.t. & all & n.t. & all & n.t. \\
\hline
\textbf{deg 1} & $\phantom{0}1.00$ & $\phantom{0}0.99$ & $\phantom{0}0.99$ & $\phantom{0}0.99$ & $\phantom{0}1.00$ & $\phantom{0}1.00$ & $\phantom{0}1.00$ & $\phantom{0}1.00$ & $\phantom{0}1.00$ & $\phantom{0}1.00$ \\
\textbf{deg 2} & $\phantom{0}0.94$ & $\phantom{0}0.88$ & $\phantom{0}1.02$ & $\phantom{0}1.02$ & $\phantom{0}0.99$ & $\phantom{0}0.99$ & $\phantom{0}0.99$ & $\phantom{0}0.99$ & $\phantom{0}1.02$ & $\phantom{0}1.02$ \\
\textbf{deg 3} & $\phantom{0}1.54$ & $\phantom{0}1.29$ & $\phantom{0}1.02$ & $\phantom{0}1.02$ & $\phantom{0}1.00$ & $\phantom{0}1.00$ & $\phantom{0}1.00$ & $\phantom{0}1.00$ & $\phantom{0}1.02$ & $\phantom{0}1.02$ \\
\textbf{deg 4} & $\phantom{0}2.15$ & $\phantom{0}1.16$ & $\phantom{0}1.02$ & $\phantom{0}1.02$ & $\phantom{0}1.01$ & $\phantom{0}1.01$ & $\phantom{0}0.99$ & $\phantom{0}0.99$ & $\phantom{0}0.99$ & $\phantom{0}0.99$ \\
\textbf{deg 5} & $\phantom{0}5.06$ & $\phantom{0}1.06$ & $\phantom{0}0.98$ & $\phantom{0}0.98$ & $\phantom{0}1.00$ & $\phantom{0}1.00$ & $\phantom{0}1.00$ & $\phantom{0}1.00$ & $\phantom{0}0.98$ & $\phantom{0}0.98$ \\
\textbf{deg 6} & $17.29$ & $\phantom{0}1.28$ & $\phantom{0}1.06$ & $\phantom{0}1.06$ & $\phantom{0}0.98$ & $\phantom{0}0.98$ & $\phantom{0}1.02$ & $\phantom{0}1.02$ & $\phantom{0}0.99$ & $\phantom{0}0.99$ \\
\textbf{deg 7} & $64.47$ & $\phantom{0}0.41$ & $\phantom{0}0.99$ & $\phantom{0}0.99$ & $\phantom{0}1.00$ & $\phantom{0}1.00$ & $\phantom{0}1.00$ & $\phantom{0}1.00$ & $\phantom{0}0.98$ & $\phantom{0}0.98$ \\
\textbf{deg 8} & \small{${>}100$} & $\phantom{0}0.85$ & $\phantom{0}1.20$ & $\phantom{0}1.20$ & $\phantom{0}0.98$ & $\phantom{0}0.98$ & $\phantom{0}1.05$ & $\phantom{0}1.05$ & $\phantom{0}1.04$ & $\phantom{0}1.04$ \\
\textbf{deg 9} & \small{${>}100$} & $\phantom{0}0.53$ & $\phantom{0}0.97$ & $\phantom{0}0.97$ & $\phantom{0}1.04$ & $\phantom{0}1.04$ & $\phantom{0}0.97$ & $\phantom{0}0.97$ & $\phantom{0}0.97$ & $\phantom{0}0.97$ \\
\hline
\end{tabular}

\caption{Statistics of Wieferich places with $q = 4$}
\label{fig:stats:q4}
\end{figure}

\begin{figure}
\begin{tabular}{|l|r|r|r|r|r|r|r|r|r|r|}
\cline{2-11}
\multicolumn{1}{c|}{} & \multicolumn{2}{c|}{\textbf{rank 1}} & \multicolumn{2}{c|}{\textbf{rank 2}} & \multicolumn{2}{c|}{\textbf{rank 3}} & \multicolumn{2}{c|}{\textbf{rank 4}} & \multicolumn{2}{c|}{\textbf{rank 5}} \\
\multicolumn{1}{c|}{} & all & n.t. & all & n.t. & all & n.t. & all & n.t. & all & n.t. \\
\hline
\textbf{deg 1} & $\phantom{0}1.00$ & $\phantom{0}1.00$ & $\phantom{0}1.00$ & $\phantom{0}1.00$ & $\phantom{0}1.01$ & $\phantom{0}1.01$ & $\phantom{0}1.00$ & $\phantom{0}1.00$ & $\phantom{0}0.98$ & $\phantom{0}0.98$ \\
\textbf{deg 2} & $\phantom{0}0.96$ & $\phantom{0}0.95$ & $\phantom{0}0.99$ & $\phantom{0}0.99$ & $\phantom{0}1.00$ & $\phantom{0}1.00$ & $\phantom{0}0.99$ & $\phantom{0}0.99$ & $\phantom{0}0.99$ & $\phantom{0}0.99$ \\
\textbf{deg 3} & $\phantom{0}0.99$ & $\phantom{0}0.95$ & $\phantom{0}0.98$ & $\phantom{0}0.98$ & $\phantom{0}1.00$ & $\phantom{0}1.00$ & $\phantom{0}0.99$ & $\phantom{0}0.99$ & $\phantom{0}1.02$ & $\phantom{0}1.02$ \\
\textbf{deg 4} & $\phantom{0}1.09$ & $\phantom{0}0.89$ & $\phantom{0}0.99$ & $\phantom{0}0.99$ & $\phantom{0}0.98$ & $\phantom{0}0.98$ & $\phantom{0}0.98$ & $\phantom{0}0.98$ & $\phantom{0}1.02$ & $\phantom{0}1.02$ \\
\textbf{deg 5} & $\phantom{0}2.08$ & $\phantom{0}1.08$ & $\phantom{0}0.98$ & $\phantom{0}0.98$ & $\phantom{0}1.00$ & $\phantom{0}1.00$ & $\phantom{0}0.99$ & $\phantom{0}0.99$ & $\phantom{0}1.02$ & $\phantom{0}1.02$ \\
\textbf{deg 6} & $\phantom{0}6.06$ & $\phantom{0}1.06$ & $\phantom{0}0.98$ & $\phantom{0}0.98$ & $\phantom{0}0.98$ & $\phantom{0}0.98$ & $\phantom{0}1.04$ & $\phantom{0}1.04$ & $\phantom{0}1.03$ & $\phantom{0}1.03$ \\
\textbf{deg 7} & $26.07$ & $\phantom{0}1.07$ & $\phantom{0}0.97$ & $\phantom{0}0.97$ & $\phantom{0}1.04$ & $\phantom{0}1.04$ & $\phantom{0}0.94$ & $\phantom{0}0.94$ & $\phantom{0}0.98$ & $\phantom{0}0.98$ \\
\textbf{deg 8} & \small{${>}100$} & $\phantom{0}0.85$ & $\phantom{0}1.01$ & $\phantom{0}1.01$ & $\phantom{0}1.05$ & $\phantom{0}1.05$ & $\phantom{0}0.99$ & $\phantom{0}0.99$ & $\phantom{0}0.95$ & $\phantom{0}0.95$ \\
\hline
\end{tabular}

\caption{Statistics of Wieferich places with $q = 5$}
\label{fig:stats:q5}
\end{figure}

Theorems~\ref{th:proba:parameter} and~\ref{th:proba:independence}
provide very precise informations on the random variables 
$W_{r,\p}$ when the rank $r$ is large enough. On the contrary,
the situation in small rank is far less clear.

In order to get a better feeling on this, we have conducted
numerical simulations for various values of $q$, $r$ (the rank)
and~$d$ (the degree). The results are reported in the tables of 
Figures~\ref{fig:stats:q2}--\ref{fig:stats:q5}. The columns
labelled ``all'' correspond to all small Drinfeld models with
the prescribed rank, while the columns labelled ``n.t.'' 
correspond only to those Drinfeld models for which $1$ is
not a torsion point. Indeed, when $1$ is torsion, every place
is Wieferich and it seems to us that this could distort the
statistics.

For each choice of the pair $(q,r)$, we sampled $10,\!000$ random 
Drinfeld models (expect if there were less than $10,\!000$, 
in which case, we have considered all of them) and reported in
each cell the empiric value of
$$\frac{q^d}{\text{Card}\: \mathcal P_d} \cdot
  \sum_{\p \in \mathcal P_d} W_{r,\p}$$
where $\mathcal P_d$ denotes the set of all places of $A$ of degree
$d$ (as in the proof of Corollary~\ref{cor:proba:independence}).
Those values are then expected to be close to $1$, at least if our 
heuristic that a place of degree $d$ is Wieferich with probability
$q^{-d}$ is correct.

We see in the tables that it is indeed the case when the degree 
remains small compared to the rank; this is in line with
Theorem~\ref{th:proba:parameter}. On the contrary, when the
degree gets higher, the behaviour looks more erratic with
entries attaining very large values. We notice nonetheless
that the bounds of Theorems~\ref{th:proba:parameter}
and~\ref{th:proba:independence} look pessimistic: the expected
behaviour seems to occur much earlier than what they claim.

Despite all of this, it is still unclear to us if expecting an
infinite number of Wieferich places for a given Drinfeld model
is reasonable or not. In any case, we emphasize that 
Theorems~\ref{th:proba:parameter} and~\ref{th:proba:independence}
do not imply such a result in average. They actually even cannot
ensure the existence of a single Drinfeld model admitting an
infinite number of Wieferich places\footnote{It is instructive
to compare with the following situation. If $\p$ is a fixed place
of degree $d$ (over $\FF$), a random polynomial of degree $r \geq
d$ is a multiple of $\p$ with probability $q^{-d}$ and those events
are independant when $r$ is large enough. However, obviously, a
given polynomial cannot be divisible by an infinite number of
places.}. 
In order to have more evidences on this question, we have run
further experiments in the special case of the Carlitz module.
Here is what we found.

\begin{itemize}
\item
For $q = 2$, all places are Wieferich expect those of degree $1$.

\item
For $q = 3$, we have looked for Wieferich places until the degree
$24$ (which corresponds to a total of $18,\!054,\!379,\!372$ places) and found
$4$ Wieferich places, namely
\begin{enumerate}[label=(\arabic{enumi}),topsep=0pt,itemsep=\parsep,parsep=0pt]
\item $t^6 + t^4 + t^3 + t^2 + 2t + 2$,
\item $t^9 + t^6 + t^4 + t^2 + 2t + 2$,
\item $t^{12} + 2t^{10} + t^9 + 2t^4 + 2t^3 + t^2 + 1$,
\item $t^{15} + t^{13} + t^{12} + t^{11} + 2t^{10} + 2t^7 
       + 2t^5 + 2t^4 + t^3 + t^2 + t + 1$.
\end{enumerate}

\item
For $q = 4$, we have looked for Wieferich places until the degree
$17$ (which corresponds to a total of $1,\!376,\!854,\!004$ places) and found
$2$ Wieferich places, namely
\begin{enumerate}[label=(\arabic{enumi}),topsep=0pt,itemsep=\parsep,parsep=0pt]
\item $t^2 + t + z$,
\item $t^2 + t + (z + 1)$,
\end{enumerate}
where $z \in \FF_4$ is a solution of $z^2 + z + 1 = 0$.

\item
For $q = 5$, we have looked for Wieferich places until the degree
$17$ (which corresponds to a total of $57,\!005,\!914,\!349$ places) and found
$2$ Wieferich places, namely
\begin{enumerate}[label=(\arabic{enumi}),topsep=0pt,itemsep=\parsep,parsep=0pt]
\item $t^5 + 4t + 1$,
\item $t^{10} + 3t^6 + 4t^5 + t^2 + t + 1$.
\end{enumerate}

\end{itemize}
Again, the conclusion is unclear.

\appendix

\section{Appendix}

In this appendix, we make the connection between Drinfeld models and their $L$-series as they are defined in the present paper and other classical definitions (in slightly different contexts) that one finds in the literature.
In particular, we establish a comparison theorem relating the $L$-series considered in this paper and those of~\cite{caruso-gazda}.

Throughout the appendix, we keep the notation of the paper; that is $\FF$ is a finite field with $q$ elements and $A=\FF[t]$.

\subsection{Models of Drinfeld modules: a geometric definition}\label{sec:models}
Let $r$ be a positive integer. In the body of the text, we defined a model of a Drinfeld $A$-module of rank $r$ over $A$ to be an $\FF$-algebra homomorphism $\phi:A\to A\{\tau\}$ satisfying
\begin{enumerate}[label=(M)]
\item\label{item:M} As a polynomial in $\tau$, $\phi(t)$ has degree $r$ and constant term $t$.
\end{enumerate}
Below, we explain why this elementary definition matches the usual geometric one. In algebraic geometry, a model of a Drinfeld module over $A$ is generally defined as an $A$-module scheme over $\Spec A$, Zariski-locally isomorphic to~$\mathbb{G}_a$, whose generic fiber is a Drinfeld module over $\FF(t)$ (see \cite[Definition 3.7]{hartl}). 
In what follows, we will prove the probably already well-known result.
\begin{prop}\label{prop:models-match}
There is a one-to-one correspondence
\[
\left\{
\begin{array}{c}
\mathrm{models~of~Drinfeld~modules} \\
\mathrm{of~rank~}r\mathrm{~over~}A
\end{array}
\right\}
~ \longleftrightarrow ~
\left\{
\begin{array}{c}
\FF\mathrm{-algebra~homomorphisms} \\
\phi:A\to A\{\tau\}\mathrm{~satisfying~}\ref{item:M}
\end{array}
\right\}.
\]
In addition, isomorphic models correspond to the same homomorphism. 
\end{prop}

We begin with general notations. Let $R$ be a (commutative, unitary) $\FF$-algebra. We let $\mathbb{G}_{a,R}$, or simply $\mathbb{G}_{a}$ for short, be the additive $\FF$-vector space scheme, \emph{i.e.} it is the functor from the category of $R$-algebras to that of $\FF$-vector spaces which forgets the $R$-algebra structure and only retain that of an $\FF$-vector space. The $q$-power map on $R$ induces an endomorphism of $\mathbb{G}_{a,R}$ which we denote by $\tau$. Any endomorphism of $\mathbb{G}_{a,R}$ can be uniquely written as a polynomial in $\tau$. That is, as rings,
\[
\End_{\FF\text{-vs}/R}(\mathbb{G}_{a,R})=R\{\tau\}
\]
(\emph{cf}. \cite[Lemma~3.2]{hartl}).
Let $K=\FF(t)$. We recall the definition of Drinfeld modules over $K$, after \cite[Definition~3.7]{hartl}.
\begin{deftn}\label{def:DM}
A \emph{Drinfeld module} $E$ over $K$ of rank $r$ is an $A$-module scheme over $K$ satisfying the following two properties:
\begin{enumerate}[label=$(\roman*)$]
\item\label{item:coordinate} As an $\FF$-vector space scheme over $K$, it is isomorphic to $\mathbb{G}_{a,K}$.
\item If $\kappa:E\stackrel{\sim}{\longrightarrow} \mathbb{G}_{a,K}$ is such an isomorphism and $\phi:A\to \End_{\FF\text{-vs}/K}(E)$ denotes the $A$-module scheme structure on $E$, the composition $\phi^{\kappa}(t):=\kappa^{-1} \circ \phi(t) \circ \kappa$ in $\End_{\FF\text{-vs}/K}(\mathbb{G}_{a,K})$ satisfies condition \ref{item:M} in $K\{\tau\}$.  
\end{enumerate}
An isomorphism $\kappa:E\stackrel{\sim}{\longrightarrow}\mathbb{G}_{a,K}$ as in \ref{item:coordinate} will be called \emph{a choice of coordinates for $E$}.
\end{deftn}

\begin{ex}
The most basic example of a Drinfeld module is the \emph{Carlitz module} which, as an $\FF$-vector space scheme, is equal to $\mathbb{G}_{a,K}$ with the $A$-module structure determined by $\phi^{\mathrm{id}}(t)=t+\tau$. 
\end{ex}

The objects of interest in this paper are not the Drinfeld modules themselves, but rather their integral models. We define them as follows. 

\begin{deftn}\label{def:model-of-DM}
Let $E$ be a Drinfeld module over $K$.
An \emph{integral model for $E$} is an $A$-module scheme $\cE$ over $A$ which is Zariski locally isomorphic to $\mathbb{G}_{a}$ as an $\FF$-vector space scheme and whose generic fiber $\cE\times_{\Spec A}\Spec \FF(t)$ is isomorphic to $E$. 
\end{deftn}

In fact, we claim that the mention ``Zariski locally'' is unnecessary in the previous definition, as long as we are over the principal ideal domain $A$.
\begin{prop}\label{prop:Ga}
Let $E$ be a Drinfeld module over $K$.
Any integral model of $E$ is isomorphic to $\mathbb{G}_a$ as an $\FF$-vector space scheme over $A$.
\end{prop}
\begin{proof}
Let $\cE$ be an integral model of $E$. By assumption, there exists a finite covering of $\Spec A$ by affine schemes $\{\Spec R_i\}_{i\in I}$, where each $R_i\subset K$ is a Zariski-localization of $A$, together with compatible isomorphisms of $\FF$-vector space schemes over each $R_i$
\[
\kappa_i:\cE\times_{\Spec A} \Spec R_i \stackrel{\sim}{\longrightarrow} \mathbb{G}_{a,R_i}. 
\]
For $(i,j)\in I^2$ we set $R_{ij}:= R_i\otimes_{A}R_j$. Note that $\Spec R_{ij}$ agrees with the intersection $(\Spec R_i)\cap (\Spec R_j)$. The composition
\[
\mathbb{G}_{a,R_{ij}}\stackrel{\kappa_j^{-1}}{\longrightarrow} E\times_{\Spec A} \Spec R_{ij} \stackrel{\kappa_i}{\longrightarrow} \mathbb{G}_{a,R_{ij}}
\]
defines an $\FF$-linear automorphism of $\mathbb{G}_{a,R_{ij}}$, hence an element $d_{ij}\in R_{ij}\{\tau\}^{\times}=R_{ij}^{\times}$. One checks that the data of $(\Spec R_{ij}, d_{ij})_{i,j\in I}$ forms a cocycle (\emph{i.e.} a Cartier divisor) in 
\[
\mathrm{H}^1_{\mathrm{Zar}}(\Spec A, \mathbb{G}_m).
\]
Yet, the latter is trivial as it identifies with the Picard group of $A$. We deduce the existence of invertible elements $d_i\in R_i^{\times}$, $i\in I$, for which $d_{ij}=d_i^{-1}d_j$. We may now perform Zariski descent over the family of isomorphisms $(d_i \kappa_i:E\times \Spec R_i\to \mathbb{G}_{a,R_i})_{i\in I}$ to obtain an isomorphism $\cE\stackrel{\sim}{\to} \mathbb{G}_{a,A}$ over $\Spec A$.
\end{proof}

We pursue by proving existence of integral models.

\begin{prop}
Any Drinfeld module $E$ over $K$ admits an integral model.
\end{prop}

\begin{proof}
We fix a choice of coordinates $\kappa:E\stackrel{\sim}{\to} \mathbb{G}_{a,K}$. The $A$-module structure on $E$ induces one on $\mathbb{G}_{a,K}$ and, as such, it induces a ring homomorphism $\phi^{\kappa}:A\to K\{\tau\}$. We let $g_i\in K$ be the $i$-th coefficient of $\phi^{\kappa}(t)$ as a polynomial in $\tau$, \emph{i.e.}
\[
\phi^{\kappa}(t)=t+g_1\tau+\cdots+g_r\tau^r.
\]
Let $d\in A$ be the least common multiple of the denominators of the $g_i$ for $i\in \{1,\ldots,r\}$. Let also $m_d$ be the multiplication by $d$ on $\mathbb{G}_{a,K}$. Then,
\[
\phi^{\kappa\circ m_d}(t)=m_d^{-1}\circ \phi^{\kappa}(t)\circ m_d=t+g_1d^{q-1}\tau+\ldots+g_r d^{q^r-1}\tau^r\in A\{\tau\}.
\] 
Because $m_d$ is an isomorphism, $\kappa':=m_d\circ \kappa$ defines an isomorphism from $E$ to $\mathbb{G}_{a,K}$ such that $\phi^{\kappa'}$ takes its values in $A\{\tau\}$. Therefore, setting $\cE := \mathbb{G}_{a,A}$ as an $\FF$-vector space scheme and equipping it with the $A$-module structure given by $\phi^{\kappa'}$ produces an integral model of~$E$. 
\end{proof}

We are in position to prove Proposition \ref{prop:models-match} stated above.
\begin{proof}[Proof of Proposition \ref{prop:models-match}]
Let $\phi:A\to A\{\tau\}$ be an $\FF$-algebra homomorphism satisfying \ref{item:M}. Let $\mathcal{E}$ be $\mathbb{G}_{a,A}$ as an $\FF$-vector space scheme and endow $\mathcal{E}$ with an $A$-module scheme structure via $\phi$, seen as a map $A\to \End_{\FF\text{-vs}/A}(\mathcal{E})$. Then $\mathcal{E}$ is an integral model (of the Drinfeld module $\mathcal{E}\times_{A}K$). 

Conversely, let $\cE$ be an integral model of a Drinfeld module in the sense of Definition \ref{def:model-of-DM}. By Proposition \ref{prop:Ga}, there exists an isomorphism $\kappa:\cE\stackrel{\sim}{\to} \mathbb{G}_{a,A}$. The $A$-module scheme structure on $\cE$ induces a ring homomorphism $\phi^{\kappa}:A\to A\{\tau\}$ which satisfies condition \ref{item:M}. We claim that $\phi^{\kappa}$ is independent of $\kappa$: indeed, another choice $\kappa'$ would differ from $\kappa$ by an element of $\mathrm{Aut}(\mathbb{G}_{a,A})=A\{\tau\}^{\times}=\FF^{\times}$. This implies $\phi^{\kappa}=\phi^{\kappa'}$. 
\end{proof}

\subsection{Drinfeld modules and Anderson motives}
\label{app:anderson}

We now make the bridge between the notation and constructions of this article with those of \cite{caruso-gazda}.
We will prove in particular a comparison theorem for their $L$-series, giving then a full justification that Theorem~\ref{thm:order} is concerned with similar objects than the main conjecture of \cite{caruso-gazda}.

\paragraph{Anderson motives.}

Let $F$ be $A$-field, that is, by definition, a field equipped with a ring homomorphism $\gamma : A \to F$.
We consider the tensor product\footnote{By convention, all unlabeled tensor products are taken over $\FF$.} $A_F := A \otimes F$; it is canonically isomorphic to $F[t]$ since $A = \Fq[t]$.
In a slight abuse of notation, we continue to denote by $t$ the element $t \otimes 1$ of $A_F$. We also set $\theta := 1 \otimes \gamma(t) \in A_F$.
We extend the Frobenius of $F$ to a endomorphism of $A$-algebras $\tau : A_F \to A_F$; hence $\tau(a \otimes x) = a \otimes x^q$ for $a \in A$ and $x \in F$.

In order to handle smoothly semi-linearity, we will often use the following classical notation:
when $M$ is a module over $A_F$, we write $\tau^* M := A_F \otimes_{A_F} M$ where $A_F$ is viewed as an algebra over itself \emph{via} $\tau$.
If $M$ and $N$ are modules over $A_F$, giving a $\tau$-semilinear map $M \to N$ is equivalent to giving an $A_F$-linear map $\tau^* M \to N$.
In a similar fashion, if $x \in M$, we write $\tau^* x$ for $1 \otimes x \in \tau^* M$.

\begin{deftn}
An \emph{Anderson motive} over $F$ is a finite free $A_F$-module $M$ equipped with a $A_F$-linear isomorphism
\[ \textstyle
\tau_M : \tau^* M\big[\frac 1{t-\theta}\big] \stackrel\sim\longrightarrow M\big[\frac 1{t-\theta}\big].
\]
\end{deftn}

One can attach an Anderson motive to any Drinfeld module $\phi=\phi^\kappa : A \to F\{\tau\}$ (after a choice of coordinates $\kappa$) using the following recipe.
We simply set $\M(\phi) := F\{\tau\}$. We equip it with the structure of $A_F$-module defined by
\[
(a \otimes x) \bullet f := x f \phi_a \qquad (a \in A, x \in F, f \in \M(\phi))
\]
where the product on the right hand side is the usual product in the ring $F\{\tau\}$.
Finally, the morphism $\tau_{\M(\phi)}$ is by definition the right multiplication by $\tau$ in $F\{\tau\}$.

It is routine to check that $\M(\phi)$ is indeed an Anderson motive.
More precisely, if we write $\phi_t = \gamma(t) + g_1 \tau + \cdots + g_r \tau^{r}$ with $g_i \in F$, $g_r \neq 0$, one proves using Euclidean division in $F\{\tau\}$ that $\M(\phi)$ is free of rank $r$ over $A_F$ and an explicit basis of it is given by $u_i := \tau^i$ for $i$ varying between $0$ and $r{-}1$ (see \cite[\S 2.2.1]{caruso-leudiere} for more details).
Moreover, the action of $\tau_{\M(\phi)}$ is explicitely given by
\begin{align}
\tau_{\M(\phi)}(\tau^* u_i) & = u_{i+1} \qquad (0 \leq i \leq r-2) \nonumber \\
\tau_{\M(\phi)}(\tau^* u_{r-1}) & = g_r^{-1} \cdot \big( (t{-}\theta) u_0 - g_1 u_1 - \cdots - g_{r-1} u_{r-1} \big). \label{eq:Mphi}
\end{align}
On this writing, we see that $\tau_{\M(\phi)}$ induces an isomorphism $\tau^* \M(\phi)\big[\frac 1{t-\theta}\big] \to \M(\phi)\big[\frac 1{t-\theta}\big]$ as desired.

\paragraph{Anderson models.}

Likewise Drinfeld models are integral versions of Drinfeld modules, Anderson models are defined as integral structures inside Anderson motives.
There are only defined for particular base fields $F$, including notably the fraction field $K$ of $A$.
In order to avoid confusion, from now on, we denote by $\theta$ the variable on $K$: we set $K := \FF(\theta)$ and endow it with a structure of $A$-field \emph{via} $\gamma : A \to K$, $t \mapsto \theta$.
The notation is then coherent with what precedes.

We also set $R := \FF[\theta]$ and $A_R := A \otimes R \simeq \FF[t,\theta]$. It is a subring of $A_K$.

\begin{deftn}
\label{def:andersonmodel}
Let $M$ be an Anderson motive over $K$.
A \emph{model} of $M$ is a finitely generated projective $A_R$-module $M_R$ which generates $M$ as an $A_K$-module and which is stable by $\tau_M$ in the following sense:
\[ \textstyle
\tau_M\big(\tau^* M_R\big) \subset M_R\big[\frac 1{t-\theta}\big].
\]
\end{deftn}

Unfortunately, the division by $g_r$ in Equation~\eqref{eq:Mphi} indicates that the $A_R$-submodule generated by the canonical basis of $\M(\phi)$ is usually not an Anderson model even if $\phi$ is itself a Drinfeld model.
One can nevertheless recover this important property by passing to the dual.

\begin{deftn}
Let $M$ be an Anderson motive over an $A$-field $F$.
The \emph{dual} of $M$ is $M^\vee := \mathrm{Hom}_{A_F}(M, A_F)$ equipped with the map $\tau_{M^\vee}$ defined as follows:
for all $f \in \tau^* M^\vee = \mathrm{Hom}_{A_F}(\tau^*M, \tau^*A_F)$, the map $\tau_{M^\vee}(f)$ is the unique one for which the following diagram commutes.
\[
\xymatrix @C=12ex {
\tau^* M\big[\frac 1{t-\theta}\big] \ar[r]^{f} \ar[d]_-{\tau_M}^-{\sim} & \tau^* A_F \big[\frac 1{t-\theta}\big] \ar[d]^-{\tau}_-{\sim} \\
M\big[\frac 1{t-\theta}\big] \ar@{-->}[r]^{\tau_{M^\vee}(f)} & A_F \big[\frac 1{t-\theta}\big]}
\]
\end{deftn}

When $M = \M(\phi)$ for a Drinfeld model $\phi : A \to R\{\tau\}$ with $\phi_t = \theta + g_1 \tau + \cdots + g_r \tau^r$, $g_r \neq 0$, a computation shows that
\begin{equation}
\label{eq:tauMphidual}
\tau_{\M(\phi)^\vee}(\tau^* u_i^\vee) = u_{i+1}^\vee + \frac{g_{i+1}}{t{-}\theta} \cdot u_0^\vee \qquad (0 \leq i < r) 
\end{equation}
where $(u_i^\vee)_{0 \leq i < r}$ is the dual basis of the canonical basis $(u_i)_{0 \leq i < r}$ of $\M(\phi)$ and where we agree by convention that $u_r^\vee = 0$.
In particular, we notice that the $M_R$-submodule of $\M(\phi)^\vee$ generated by the $u_i^\vee$ defines a model in the sense of Definition~\ref{def:andersonmodel}.

This construction defines a rank-preserving covariant functor
\begin{equation}
\label{eq:Mvee}
\M^\vee : 
\big\{\, \mathrm{Drinfeld~models~over~}R \,\big\}
\longrightarrow
\big\{\, \mathrm{Anderson~models~over~}R \,\big\}.
\end{equation}

\paragraph{Euler factors.}

We assume that $F$ is a finite field and we let $\p$ be the monic generator of the ideal $\ker(\gamma : A \to F)$.
We set $d := \deg \p$, it is also the degree of the extension $F/\FF$.
If $M$ is an Anderson motive over $F$, we define following~\cite{caruso-gazda}
\begin{equation}
\label{eq:eulerfactor} \textstyle
P(M; T) := \mathrm{det}_{A[\frac 1{\p},T] \otimes F}\big(1 - T^d \tau_M^d \,|\, A\big[\frac 1 \p,T\big] \otimes_A M\big).
\end{equation}
We notice that $\tau_M^d$ indeed induces an endomorphism of $A[\frac 1 \p] \otimes_A M$ because of a combinaison of two facts:
first, $t{-}\theta$ is invertible in the ring $A[\frac 1 \p] \otimes_A F$ and second, $\tau^d$ is the identity on $A_F$.
It can be shown moreover that $P(M; T)$ is a polynomial with coefficients in $A[\frac 1 \p]$. We call it the \emph{Euler factor} of $M$;
it will serve as a local factor when we will define $L$-series later on.

Alternatively, $P(M; T)$ can be seen as another determinant involving the dual of $M$ instead of $M$ itself.
Indeed, remarking that $\tau_{M^\vee}^d$ is the inverse of the dual map of $\tau_M^d$, we get the relation
\begin{equation}
\label{eq:eulerdual}
P(M; T) = 
\frac{\mathrm{det}_{A[\frac 1{\p},T] \otimes F}\big(\tau_{M^\vee}^d - T^d \,|\, A\big[\frac 1 \p,T\big] \otimes_A M^\vee\big)}
     {\mathrm{det}_{A[\frac 1{\p}] \otimes F}\big(\tau_{M^\vee}^d \,|\, A\big[\frac 1 \p\big] \otimes_A M^\vee\big)}.
\end{equation}
In the above fraction, the numerator is the so-called \emph{characteristic polynomial of the Frobenius} of $M^\vee$ (up to a sign)
whereas the denominator serves as a renormalization to ensure that the constant coefficient of $P(M; T)$ is $1$.

\begin{thm}
\label{thm:eulerfactor}
Let $\phi : A \to F\{\tau\}$ be a Drinfeld module and let $\phiT$ be its $T$-twisted version as defined in Subsection~\ref{sssec:Ttwisted}.
Then
\[
P\big(\M(\phi)^\vee; T\big) = \p^{-1} \cdot |\phiT(F)|.
\]
\end{thm}

\begin{proof}
After Equation~\eqref{eq:eulerdual} and the fact that the constant coefficient of both $P\big(\M(\phi)^\vee; T\big)$ and $\p^{-1}{\cdot}|\phiT(F)|$ is $1$, it is enough to prove that
\begin{equation}
\label{eq:collinear}
\textstyle
\mathrm{det}_{A_F[T]}\big(T^d - \tau_{\M(\phi)}^d \,|\, \M(\phi)[T]\big)
\text{ is $A_F$-collinear to } |\phiT(F)|.
\end{equation}
We consider the Ore polynomial ring $A_F\{\tau\}$. 
Since $\tau^d$ acts on $A_F$ as the identity, $A_F\{\tau\}$ is an Azumaya algebra over its centre $A[\tau^d]$ and we have a reduced norm map
$\Nrd : A_F\{\tau\} \to A[\tau^d]$.
By \cite[Theorem~4.5]{caruso-leudiere} (see also \cite[Remark~4.6]{caruso-leudiere}), the left-hand side in~\eqref{eq:collinear},
which is the characteristic polynomial of the Frobenius of $\phi$,
is collinear to $\Nrd(t - \phi_t)_{|\tau^d = T^d}$.

We now need to connect what precedes to $|\phiT(F)|$.
For this, our starting point is the $A[T]$-linear exact sequence~\eqref{eq:resolutionphiT} that we recall below:
\[
0\longrightarrow A_F[T] \xrightarrow{t-\phiT_t} A_F[T] \xrightarrow{a \otimes x\:\mapsto\:\phiT_a(x)} \phiT(F) \longrightarrow 0.
\]
There is an obvious bijection $\alpha : A_F[T] \to A_F\{\tau\}$ simply obtained by mapping a polynomial $\sum_i a_i T^i$ (with $a_i \in A_F$) to $\sum_i a_i \tau^i$.
Of course, $\alpha$ is not a ring homomorphism. Nonetheless, it is $A$-linear and even $A[T]$-linear if we let $T$ act on $A_F\{\tau\}$ by right-multiplication by $\tau$.
Moreover, a simple computation shows that we have a commutative diagram
\[
\xymatrix @C=10ex {
A_F[T] \ar[r]^-{t - \phiT_t} \ar[d]_-{\alpha} & A_F[T] \ar[d]^-{\alpha} \\
A_F\{\tau\} \ar[r]^-{\mu} & A_F\{\tau\} 
}
\]
where the bottom map $\mu$ is the left-multiplication by $t - \phi_t$. Hence
\[
|\phiT(F)| 
 = \mathrm{det}_{A[T]}\big(t - \phiT_t \, | \, A_F[T]\big) 
 = \mathrm{det}_{A[T]}\big(\mu \, | \, A_F\{\tau\}\big) = \Nrd(t - \phi_t)_{|\tau^d = T^d}.
\]
The theorem is proved.
\end{proof}

\paragraph{$L$-series.}

We are finally ready to relate the $L$-series attached to Drinfeld models to those attached to models of Anderson motives.
First of all, we recall the definition of the latter.
We start with a model $M_R$ of an Anderson motive $M$ over $K = \FF(\theta)$.
For a place $\q$ of $R$, \emph{i.e.} a monic irreducible polynomial $\q \in R$, we form the quotient $M_\q = M_R / \q M_R$.
It is a module over $A_{\FF_\q}$ with $\FF_\q := R/\q$. Moreover, $\tau_M$ induces a map
\[ \textstyle
\tau^* M_\q \big[\frac 1{t-\theta}\big] \longrightarrow M_\q\big[\frac 1{t-\theta}\big].
\]
The latter is not necessarily an isomorphism, meaning that $M_\q$ might fail to be an Anderson motive over $\FF_\q$.
Nevertheless, the formula of Equation~\eqref{eq:eulerfactor} still makes sense and we can reuse it to defined the local factor at $\q$:
\[ \textstyle
P_\q(M_R; T) := \mathrm{det}_{A[\frac 1{\q},T] \otimes \FF_\q}\big(1 - T^d \tau_M^d \,|\, A\big[\frac 1 \q,T\big] \otimes_A M_\q\big).
\]
Finally, the $L$-series of $M_R$ is obtained by combining the local factors as follows:
\[
L(M_R; T) := \prod_{\q} \frac 1 {P_\q(M_R; T)}
\]
where the product runs over all the places $\q$ of $R$.
Similarly, the $\p$-adic $L$-series of $M_R$ is
\[
L_\p(M_R; T) := \prod_{\q \neq \p} \frac 1 {P_\q(M_R; T)}.
\]

\begin{thm}
\label{thm:compLseries}
Let $\phi : A \to R\{\tau\}$ be a Drinfeld model and let $\M^\vee(\phi)$ be the Anderson model attached to it by the functor~\eqref{eq:Mvee}.
\begin{enumerate}[label=(\theenumi)]
\item We have $L(\phi; T) = L\big(\M^\vee(\phi); T\big)$.
\item We have $L_\p(\phi; T) = L_\p\big(\M^\vee(\phi); T\big)$ for any place $\p$ of $R$.
\end{enumerate}
\end{thm}

\begin{proof}
We write $M_R := \M^\vee(\phi)$ for simplicity.
Let $\q$ be a place of $R$ and let $\psi : A \to \FF_\q\{\tau\}$ be the reduction of $\phi$ modulo $\q$.
The action of $\tau_M$ on $M_R$, and hence on $M_\q$, is given by the formula~\eqref{eq:tauMphidual}.
Let $M_{\q,\nil}$ be the submodule of $M_\q$ generated by $u_s^\vee, \ldots, u_{r-1}^\vee$ where $s$ is the rank of $\psi$.
We have $M_{\q} / M_{\q,\nil} \simeq \M(\psi)^\vee$. Moreover, $\tau_M$ induces a nilpotent action on $M_{\q,\nil}$. It follows that
\[ \textstyle
P_\q(M_R; T) 
  = \mathrm{det}_{A[\frac 1{\q},T] \otimes \FF_\q}\big(1 - T^d \tau_M^d \,|\, A\big[\frac 1 \q,T\big] \otimes_A \M(\psi)^\vee\big)
  = P\big(\M(\psi)^\vee; T\big).
\]
Applying finally Theorem~\ref{thm:eulerfactor}, we end up with $P_\q(M_R; T) = P_\q(\phi; T)$
and the theorem follows by taking the product over all the relevant places $\q$.
\end{proof}

\end{document}